\documentclass[psamsfonts]{amsart}

\usepackage{amssymb,amsfonts}
\usepackage[all,arc]{xy}
\usepackage{enumerate}
\usepackage{mathrsfs}
\usepackage{url}
\usepackage[T1]{fontenc}

\newtheorem{thm}{Theorem}[section]
\newtheorem{cor}[thm]{Corollary}
\newtheorem{prop}[thm]{Proposition}
\newtheorem{lem}[thm]{Lemma}

\theoremstyle{definition}
\newtheorem{defn}[thm]{Definition}

\newtheorem{exmp}[thm]{Example}

\newtheorem{fact}[thm]{Fact}

\theoremstyle{remark}
\newtheorem{rem}[thm]{Remark}

\newcommand{\tp}{\mathrm{tp}}
\newcommand{\acl}{\mathrm{acl}}
\newcommand{\dcl}{\mathrm{dcl}}
\newcommand{\qftp}{\mathrm{qftp}}
\newcommand{\diag}{\mathrm{diag}}
\newcommand{\diagf}{\mathrm{diag}_f}
\newcommand{\Aut}{\mathrm{Aut}}
\newcommand{\NSOP}[1]{\mathrm{NSOP}_{#1}}
\newcommand{\SOP}[1]{\mathrm{SOP}_{#1}}
\newcommand{\Sk}{\mathrm{Sk}}

\makeatletter
\let\c@equation\c@thm
\makeatother
\numberwithin{equation}{section}

\def\Ind#1#2{#1\setbox0=\hbox{$#1x$}\kern\wd0\hbox to 0pt{\hss$#1\mid$\hss}\lower.9\ht0\hbox to 0pt{\hss$#1\smile$\hss}\kern\wd0}
\def\Notind#1#2{#1\setbox0=\hbox{$#1x$}\kern\wd0\hbox to 0pt{\mathchardef
	\nn=12854\hss$#1\nn$\kern1.4\wd0\hss}\hbox to
	0pt{\hss$#1\mid$\hss}\lower.9\ht0 \hbox to
	0pt{\hss$#1\smile$\hss}\kern\wd0}
\newcommand{\ind}[1][]{\mathop{\mathpalette\Ind{}^{\!\!\!\!\rlap{$\scriptscriptstyle{#1}$}\,\,\,\,}}}
\newcommand{\nind}[1][]{\mathop{\mathpalette\Notind{}^{\!\!\!\rlap{$\scriptscriptstyle{#1}$}\,\,\,}}}

\title{Generic expansion and Skolemization in $\NSOP{1}$ theories}

\date{\today}

\begin{document}
\begin{abstract}
We study expansions of $\NSOP{1}$ theories that preserve $\NSOP{1}$.  We prove that if $T$ is a model complete $\NSOP{1}$ theory eliminating the quantifier $\exists^{\infty}$, then the generic expansion of $T$ by arbitrary constant, function, and relation symbols is still $\NSOP{1}$. We give a detailed analysis of the special case of the theory of the generic $L$-structure, the model companion of the empty theory in an arbitrary language $L$. Under the same hypotheses, we show that $T$ may be generically expanded to an $\NSOP{1}$ theory with built-in Skolem functions. In order to obtain these results, we establish  strengthenings of several properties of Kim-independence in $\NSOP{1}$ theories, adding instances of algebraic independence to their conclusions.
\end{abstract}
\author{Alex Kruckman and Nicholas Ramsey}
\maketitle
\section{Introduction}\label{sec:intro}

Many of the early developments in the study of simple theories were guided by the thesis that a simple theory can be understood as a stable theory plus some `random noise.'  This loose intuition became a concrete recipe for creating new simple theories:  start with a stable theory, and, through some kind of generic construction, add additional random structure in an expanded language.  This strategy was pursued by Chatzidakis and Pillay \cite{chatzidakis1998generic}, who showed that adding a generic predicate or a generic automorphism to a stable theory results in a simple theory which is, in general, unstable.  In the case of adding a generic predicate, it suffices to assume that the base theory is simple; that is, expansion by a generic predicate preserves simplicity. The paper \cite{chatzidakis1998generic} spawned a substantial literature on generic structures and simple theories, which in turn shed considerable light on what a general simple theory might look like.  

We are interested in using generic constructions to produce new examples of $\NSOP{1}$ theories.  The class of $\NSOP{1}$ theories, which contains the class of simple theories, was isolated by D\v{z}amonja and Shelah \cite{dvzamonja2004maximality} and later investigated by Shelah and Usvyatsov \cite{shelah2008more}.  Until recently, very few non-simple examples were known to lie within this class.  A criterion, modeled after the well-known theorem of Kim and Pillay characterizing the simple theories as those possessing a well-behaved independence relation, was observed by Chernikov and the second-named author in \cite{ArtemNick}.  This criterion was applied to show that the theory of an $\omega$-free PAC field of characteristic zero and the theory of an infinite dimensional vector space over an algebraically closed field with a generic bilinear form are both $\NSOP{1}$.  That paper also showed, by a variation on a construction of Baudisch, that a simple theory obtained as a Fra\"iss\'e limit with no algebraicity may be `parametrized' to produce an $\NSOP{1}$ theory which is, in general, non-simple.  Later, Kaplan and the second-named author developed a general theory of independence in $\NSOP{1}$ theories, called \emph{Kim-independence}, which turns out to satisfy many of the familiar properties of forking independence in simple theories (e.g.\ extension, symmetry, the independence theorem, etc.) \cite{KRKim}.  In this paper, we apply this theory of independence to verify that certain generic constructions preserve $\NSOP{1}$.  

In Section~\ref{sec:nsop1}, we review the theory of Kim-independence in $\NSOP{1}$ theories and make some technical contributions to this theory. We establish strengthenings of the extension property, the chain condition, and the independence theorem for Kim-independence, obtaining additional instances of algebraic independence in their conclusions (see Definition~\ref{def:algreas}, and Theorems~\ref{thm:algreasext}, \ref{thm:algreaschain}, and \ref{thm:algreasind}). The main deficiency of Kim-independence is the failure of base monotonicity, and this work can be viewed as an effort to circumvent that deficiency, since the instances of algebraic independence that we need would be automatic in the presence of base monotonicity (see Remarks~\ref{rem:simple} and~\ref{rem:basemon}).

Section~\ref{sec:emptytheory} is dedicated to an analysis of the theory $T^\emptyset_L$ of the generic $L$-structure (the model completion of the empty theory in an arbitrary language $L$).  The work in this section was motived by a preprint of Je\v{r}\'{a}bek \cite{jerabek}. In an early draft of~\cite{jerabek}, Je\v{r}\'{a}bek showed that $T^\emptyset_L$ is always $\NSOP{3}$, regardless of the language. He asked if this could be improved to $\NSOP{1}$ and if $T^\emptyset_L$ weakly eliminates imaginaries. We give positive answers to these questions, and we characterize Kim-independence and forking independence in this theory. In a subsequent draft of~\cite{jerabek}, Je\v{r}\'{a}bek also independently answered both questions.

But Je\v{r}\'{a}bek's first question suggested a much more general one.  An $L$-theory $T$ may be considered as an $L'$-theory for any language $L'$ that contains $L$.  A theorem of Winkler \cite{winkler1975model} establishes that, as an $L'$-theory, the theory $T$ has a model completion $T_{L'}$, provided that $T$ is model complete and eliminates the quantifier $\exists^{\infty}$.  The theory $T_{L'}$ axiomatizes the generic expansion of $T$ by the new constants, functions, and relations of $L'$.  Using the theory developed in Section~\ref{sec:nsop1}, we show that if $T$ is $\NSOP{1}$, then $T_{L'}$ is as well; that is, generic expansions preserve $\NSOP{1}$.

In \cite{winkler1975model}, Winkler also showed that if $T$ is a model complete theory eliminating the quantifier $\exists^{\infty}$, then $T$ has a generic Skolemization $T_{\Sk}$.  More precisely, if $T$ is an $L$-theory, one may expand the language by adding a function symbol $f_{\varphi}$ for each formula $\varphi(\overline{x},y)$ of $L$.  And $T$,  together with axioms asserting that each $f_{\varphi}(\overline{x})$ acts as a Skolem function for $\varphi(\overline{x},y)$, has a model companion.  This result was used by N\"{u}bling in \cite{nubling2004adding}, who showed that one may Skolemize \emph{algebraic} formulas in a simple theory while preserving simplicity.  N\"ubling further observed that, in general, adding a generic Skolem function for a non-algebraic formula produces an instance of the tree property, and hence results in a non-simple theory.  We show, however, that generic Skolemization preserves $\NSOP{1}$.  By iterating, we show that any $\NSOP{1}$ theory eliminating the quantifier $\exists^{\infty}$ can be expanded to an $\NSOP{1}$ theory with built-in Skolem functions, and we also characterize Kim-independence in the expansion in terms of Kim-independence in the original theory. This result is of intrinsic interest, but it also provides a new technical tool in the study of Kim-independence in $\NSOP{1}$ theories, which, at least at its current stage of development, only makes sense when the base is a model.  Preservation of $\NSOP{1}$ by generic expansion and generic Skolemization is established in Section~\ref{sec:expsko}.  

\section{$\NSOP{1}$ and independence}\label{sec:nsop1}

\subsection{Preliminaries on $\NSOP{1}$}

Throughout this section, we fix a complete theory $T$ and a monster model $\mathbb{M}\models T$. 

\begin{defn}
A formula $\varphi(x;y)$ has $\SOP{1}$ modulo $T$ if there is a tree of tuples $(a_{\eta})_{\eta \in 2^{<\omega}}$ in $\mathbb{M}$ so that:
\begin{itemize}
\item For all $\eta \in 2^{\omega}$, the partial type $\{\varphi(x;a_{\eta | \alpha}) \mid \alpha < \omega\}$ is consistent.
\item For all $\nu,\eta \in 2^{<\omega}$, if $\nu \frown \langle 0 \rangle \unlhd \eta$ then $\{\varphi(x;a_{\eta}),\varphi(x;a_{\nu \frown \langle 1 \rangle})\}$ is inconsistent.  
\end{itemize}
The theory $T$ is $\NSOP{1}$ if no formula has $\SOP{1}$ modulo $T$.  An incomplete theory is said to be $\NSOP{1}$ if every completion is $\NSOP{1}$.  
\end{defn}

\begin{defn}
We call any $p \in S(\mathbb{M})$ a \emph{global type}.  A global type $p$ is $A$\emph{-invariant} if, for all formulas $\varphi(x;y)$, if $b \equiv_{A} b'$, then $\varphi(x;b) \in p$ if and only if $\varphi(x;b') \in p$ (equivalently, $p$ is invariant under the action of $\text{Aut}(\mathbb{M}/A)$ on $S(\mathbb{M})$).  If $p$ is a global $A$-invariant type, a \emph{Morley sequence in $p$ over $A$} is a sequence $(b_{i})_{i \in I}$ from $\mathbb{M}$ so that $b_{i} \models p|_{A(b_{j})_{j < i}}$. We denote by $p^{\otimes n}|_A$ the type $\tp(b_{i_1},\dots,b_{i_{n}}/A)$ when $i_1<i_2<\dots<i_n$. By invariance, this type does not depend on the choice of Morley sequence $(b_i)_{i\in I}$ or indices $i_k$.
\end{defn}

\begin{defn}
Fix a model $M\prec \mathbb{M}$.  
\begin{enumerate}
\item A formula $\varphi(x;b)$ \emph{Kim-divides} over $M$ if there is an $M$-invariant global type $q \supseteq \text{tp}(b/M)$ so that if $(b_{i})_{i <\omega}$ is a Morley sequence over $M$ in $q$, then $\{\varphi(x;b_{i}) \mid i <\omega \}$ is inconsistent. 
\item A partial type $p(x)$ \emph{Kim-divides} over $M$ if $p(x)$ implies some formula which Kim-divides over $M$.   
\item We write $a \ind[K]_{M} b$ for the assertion that $\text{tp}(a/Mb)$ does not Kim-divide over $M$. 
\end{enumerate} 
\end{defn}

A well-known theorem of Kim and Pillay characterizes the simple theories as those theories with a notion of independence satisfying certain properties\textemdash this serves both as a useful way to establish that a theory is simple and as a method to characterize forking independence for the given theory.  An analogous criterion for establishing that a theory is $\NSOP{1}$ was proved in \cite{ArtemNick}.  Later, the theory of Kim-independence was developed and it was observed in \cite{KRKim} that this criterion gives rise to an abstract characterization of $\ind[K]$.

\begin{thm} \label{criterion} \cite[Proposition 5.8]{ArtemNick} \cite[Theorem 9.1]{KRKim}
Assume there is an \(\Aut(\mathbb{M})\)-invariant ternary relation \(\ind\) on small subsets of \(\mathbb{M}\) satisfying the following properties, for an arbitrary $M\prec \mathbb{M}$ and arbitrary tuples from $\mathbb{M}$:
\begin{enumerate}
\item Strong finite character:  if \(a \nind_{M} b\), then there is a formula \(\varphi(x,b,m) \in \text{tp}(a/Mb)\) such that for any \(a' \models \varphi(x,b,m)\), \(a' \nind_{M} b\).
\item Existence over models: \(a \ind_{M} M\).
\item Monotonicity: if \(aa' \ind_{M} bb'\), then \(a \ind_{M} b\).
\item Symmetry: if \(a \ind_{M} b\), then \(b \ind_{M} a\).
\item The independence theorem: if \(a \ind_{M} b\), \(a' \ind_{M} c\), \(b \ind_{M} c\) and $a \equiv_{M} a'$, then there exists $a''$ with $a'' \equiv_{Mb} a$, $a'' \equiv_{Mc} a'$, and $a'' \ind_{M} bc$.
\end{enumerate}
Then \(T\) is NSOP\(_{1}\) and $\ind$ strengthens $\ind[K]$, i.e.\ if $a \ind_{M} b$, then $a \ind[K]_{M} b$.  If $\ind$ satisfies the following additional property, then $\ind = \ind[K]$ over models, i.e.\ $a\ind_M b$ if and only if $a\ind[K]_M b$:
\begin{enumerate}
\addtocounter{enumi}{5}
\item Witnessing:  if $a \nind_{M} b$, then there exists a formula $\varphi(x,b,m)\in \tp(a/Mb)$, such that for any Morley sequence $(b_{i})_{i <\omega}$ over $M$ in a global $M$-finitely satisfiable type extending $\text{tp}(b/M)$, $\{\varphi(x,b_{i},m) \mid i <\omega\}$ is inconsistent. 
\end{enumerate}
\end{thm}

\begin{thm} \label{Kimprops} \cite{KRKim}
If $T$ is $\NSOP{1}$, then $\ind[K]$ satisfies the properties (1)-(6) in Theorem~\ref{criterion}, as well as 
\begin{enumerate}
\setcounter{enumi}{6}
\item Extension: if $a\ind_M b$, then for any $c$, there exists $a'$ such that $a'\equiv_{Mb} a$ and $a'\ind_M bc$.
\item The chain condition: if $a\ind_M b$ and $I = (b_i)_{i<\omega}$ is a Morley sequence over $M$ in a global $M$-invariant type extending $\tp(b/M)$, then there exists $a'$ such that $a'\equiv_{Mb} a$, $a'\ind_M I$, and $I$ is $Ma'$-indiscernible.
\end{enumerate} 
\end{thm}

We will also be interested in the relation of algebraic independence, $\ind[a]$. Algebraic independence comes close to satisfying the criteria in Lemma~\ref{criterion} in any theory, but it typically does not satisfy the independence theorem.

\begin{defn}
For any set $C\subset \mathbb{M}$ and any tuples $a$ and $b$, we define $$a\ind[a]_C b \iff \acl(Ca)\cap \acl(Cb) = \acl(C).$$
\end{defn}

\begin{lem} \label{algprops}
In any theory $T$, $\ind[a]$ satisfies extension, existence over models, monotonicity, symmetry, strong finite character, and witnessing.  
\end{lem}

\begin{proof}
Extension for algebraic independence is proved in \cite[Theorem 6.4.5]{hodges1993model}.  Existence over models, monotonicity, and symmetry are immediate from the definitions.  

For strong finite character and witnessing, note that if $M \models T$, $a \nind^{a}_{M} b$, and $c \in (\text{acl}(Ma) \cap \text{acl}(Mb) ) \setminus M$ witnesses this, then one can choose $\chi(z;a,m) \in \text{tp}(c/Ma)$ and $\psi(z;b,m)\in \tp(c/Mb)$ which isolate these types. In particular, for some $k,k'<\omega$, we may choose $\chi$ and $\psi$ so that for all $a'$, there are at most $k$ realizations of $\chi(z;a',m)$, and for all $b'$, there are at most $k'$ realizations of $\psi(z;b',m)$. Note that, since $\psi(z;b,m)$ isolates $\tp(c/Mb)$, if $b'\equiv_M b$, then none of the realizations of $\psi(z;b',m)$ are in $M$. 

Let $\varphi(x,b,m)$ be the formula $\exists z\, (\chi(z;x,m)\land \psi(z;b,m))$. For any $a'$ satisfying $\varphi(x,b,m)$, the witness to the existential quantifier also witnesses $a'\nind[a]_M b$. This verifies strong finite character.

We use the same formula $\varphi(x,b,m)$ for witnessing. Let $(b_{i})_{i<\omega}$ be a Morley sequence over $M$ in a global $M$-finitely satisfiable type extending $\tp(b/M)$. If $\mathbb{M}\models \exists z\, (\psi(z;b_{i},m) \wedge \psi(z;b_{j},m))$ with $i<j$, then by finite satisfiability, there exists $m' \in M$ such that $\mathbb{M}\models \exists z \,(\psi(z;b_{i},m) \wedge \psi(z;m',m))$. But every realization of $\psi(z;m',m)$ is algebraic over $M$ and hence in $M$, while no realization of $\psi(z; b_i,m)$ is in $M$. It follows that the sets $\{\psi(\mathbb{M};b_i,m)\mid i<\omega\}$ are pairwise disjoint, and thus the partial type $\{\varphi(x,b_i,m)\mid i<\omega\}$ is inconsistent, since for any $a'$, the set $\chi(\mathbb{M};a',m)$ intersects at most $k$ of the sets $\psi(\mathbb{M};b_i,m)$.
\end{proof}

In Section~\ref{sec:expsko}, we will need strengthenings of the extension property and the independence theorem, which tell us that Kim-independence interacts with algebraic independence in a reasonable way. Along the way to proving the strengthening of the independence theorem, we will need a similar strengthening of the chain condition.

\begin{defn}\label{def:algreas}
Using the same notation as in Theorem~\ref{criterion}, we define the following properties of an abstract independence relation $\ind$:
\begin{itemize}
\item \emph{Algebraically reasonable extension}: if $a\ind_M b$, then for any $c$, there exists $a'$ such that $a'\equiv_{Mb} a$ and $a'\ind_M bc$, and further $a'\ind[a]_{Mb} c$. 
\item \emph{The algebraically reasonable chain condition}: if $a\ind_M b$ and $I = (b_i)_{i<\omega}$ is a Morley sequence over $M$ in a global $M$-invariant type extending $\tp(b/M)$, then there exists $a'$ such that $a'\equiv_{Mb} a$, $a'\ind_M I$, and $I$ is $Ma'$-indiscernible, and further $b_i\ind[a]_{Ma'} b_j$ for all $i\neq j$. 
\item \emph{The algebraically reasonable independence theorem}: if $a\ind_M b$, $a'\ind_M c$, $b\ind_M c$, and $a\equiv_M a'$, then there exists $a''$ such that $a''\equiv_{Mb} a$, $a''\equiv_{Mc} a'$, and $a''\ind_M bc$, and further $a''\ind[a]_{Mb} c$, $a''\ind[a]_{Mc} b$, and $b\ind[a]_{Ma''} c$. 
\end{itemize}
\end{defn}

\begin{rem}\label{rem:simple}
If $T$ is simple, then by \cite[Proposition 8.4]{KRKim}, $\ind[K]$ coincides with forking independence $\ind[f]$ over models, and the ``and further'' clauses of Definition~\ref{def:algreas} follow easily from the basic properties of forking independence.
\end{rem}

\begin{rem}\label{rem:basemon}
In \emph{any} theory, forking independence satisfies base monotonicity and strengthens algebraic independence. So for any set $A$ and tuples $a$, $b$, and $c$, $a \ind[f]_{A} bc$ implies $a \ind[f]_{Ab} c$, which implies $a \ind[a]_{Ab} c$.  Even in an $\NSOP{1}$ theory $T$, however, it is possible to have a model $M \models T$ and tuples $a$, $b$, and $c$, with $a \ind[K]_{M} bc$ and $a \nind^{a}_{Mb} c$.  See Example~\ref{ex:binaryfct} below.
\end{rem}

In the remainder of this section, we will show that in an $\NSOP{1}$ theory, Kim-independence satisfies the algebraically reasonable properties in Definition~\ref{def:algreas}. The reader who is not interested in the technicalities of these proofs  may skip directly to Section~\ref{sec:emptytheory}.

\subsection{An improved independence theorem}

We will first establish a slight improvement to the conclusion of the independence theorem, removing the apparent asymmetry between $a$, $b$, and $c$ in the conclusion. As in Remark~\ref{rem:simple}, this improved statement is easy in the context of a simple theory, where it follows from the basic properties of forking independence.

\begin{defn}\label{tmsdefs}
Suppose $T$ is $\NSOP{1}$, $M \prec\mathbb{M}$, and $(a_{i})_{i < \omega}$ is an $M$-indiscernible sequence.
\begin{enumerate}
\item Say $(a_{i})_{i < \omega}$ is a \emph{witness} for Kim-dividing over $M$ if, whenever $\varphi(x;a_{0})$ Kim-divides over $M$, $\{\varphi(x;a_{i}) \mid i <\omega\}$ is inconsistent.
\item Say $(a_{i})_{i < \omega}$ is a \emph{strong witness} to Kim-dividing over $M$ if, for all $n <\omega$, the sequence $(a_{n \cdot i}, a_{n \cdot i + 1}, \ldots, a_{n \cdot i + n-1})_{i < \omega}$ is a witness to Kim-dividing over $M$.  
\item If $I$ is any ordered index set, we say $(a_i)_{i\in I}$ is a \emph{strong witness} to Kim-dividing over $M$ if it has the same EM-type as a strong witness to Kim-dividing over $M$ indexed by $\omega$.
\end{enumerate}
\end{defn}

By \cite[Proposition 7.9]{KRKim}, in an NSOP$_{1}$ theory, the class of strong witnesses to Kim-dividing over $M$ coincides with the \emph{tree Morley sequences} over $M$.  As we will not need the Morley tree machinery of \cite{KRKim}, we will refer only to strong witnesses.  

\begin{fact} \label{tmsfacts}
Suppose $T$ is $\NSOP{1}$ and $M \models T$.  
\begin{enumerate}
\item Suppose $(a_{i},b_{i})_{i \in I}$ is a strong witness to Kim-dividing over $M$ and $J \subseteq I$ is an infinite subset.  Then $(a_{i})_{i \in J}$ and $(b_{i})_{i \in J}$ are strong witnesses to Kim-dividing over $M$.  \cite[Lemma 5.9]{KRKim}
\item If $b \equiv_{M} b'$ and $b \ind[K]_{M} b'$, then there is a strong witness to Kim-dividing over $M$, $(b_{i})_{i \in \mathbb{Z}}$, with $b_{0} = b$ and $b_{1} = b'$.  \cite[Corollary 6.6]{KRKim}
\item $a \ind[K]_{M} b$ if and only if there is an $Ma$-indiscernible sequence $(b_{i})_{i < \omega}$ which is a strong witness to Kim-dividing over $M$ with $b_{0} = b$.  \cite[Corollary 5.14]{KRKim}
\item If $a \ind[K]_{M} b$, and $I = (b_i)_{i\in \mathbb{Z}}$ is a strong witness to Kim-dividing over $M$ with $b_0 = b$, then there exists $a'\equiv_{Mb} a$ such that $I$ is $Ma'$-indiscernible and $a' \ind[K]_{M} I$.  \cite[Corollary 5.15]{KRKim}
\end{enumerate}
\end{fact}

\begin{thm} \label{strongit}
Suppose $T$ is $\NSOP{1}$, $M\prec \mathbb{M}$, \(a_0 \ind[K]_{M} b\), \(a_1 \ind[K]_{M} c\), \(b \ind[K]_{M} c\) and $a_0 \equiv_{M} a_1$. Then there exists $a$ with $a \equiv_{Mb} a_0$, $a \equiv_{Mc} a_1$, and $a \ind[K]_{M} bc$, and further $b\ind[K]_M ac$ and $c\ind[K]_M ab$.
\end{thm}

\begin{proof}
Applying the independence theorem, we obtain $a_2$ with $a_2\equiv_{Mb} a_0$, $a_2\equiv_{Mc} a_1$, and $a_2\ind[K]_M bc$. Since $b\ind[K]_M c$, by extension and an automorphism, there exists $b'$ with $b'\equiv_{Mc} b$ such that $b\ind[K]_M b'c$. By symmetry, $b'c\ind[K]_M b$, and by extension and an automorphism again, there exists $c'$ with $c'\equiv_{Mb} c$ such that $b'c\ind[K]_M bc'$. Altogether, $b'c \equiv_M bc \equiv_M bc'$, so by Fact~\ref{tmsfacts}(2), there is a strong witness to Kim-dividing over $M$, $I = (b_{i},c_{i})_{i \in \mathbb{Z}}$, so that $(b_{0},c_{0}) = (b,c')$ and $(b_{1},c_{1}) = (b',c)$.

Choose $a_3$ such that $a_3bc'\equiv_M a_2bc$. Then $a_3\ind[K]_M bc'$, so by Fact \ref{tmsfacts}(4), there exists $a \equiv_{Mbc'} a_3$ such that $I$ is $Ma$-indiscernible and $a \ind[K]_{M} I$.  By monotonicity, $a\ind[K]_M bc$. And we have $a\equiv_{Mb} a_3 \equiv_{Mb} a_2 \equiv_{Mb} a_0$, and by indiscernibility, $ac\equiv_M ac'\equiv_M a_3c' \equiv_M a_2c \equiv_M a_1c$, so $a\equiv_{Mc} a_1$.

By Fact \ref{tmsfacts}(1), $(b_{i})_{i \leq 0}$ is a strong witness to Kim-dividing over $M$ with $b_{0} = b$ which is $Mac$-indiscernible, so $b \ind[K]_{M} ac$ by symmetry and Fact~\ref{tmsfacts}(3).  Likewise, $(c_{i})_{i \geq 1}$ is a strong witness to Kim-dividing over $M$ with $c_{1} = c$ which is $Mab$-indiscernible, so $c \ind[K]_{M} ab$.  This completes the proof.  
\end{proof}

As an immediate corollary of the strengthened independence theorem, we get the following form of extension.

\begin{cor} \label{easycorollary}
Suppose $T$ is $\NSOP{1}$ and $M \models T$.  If $a \ind[K]_{M} b$ and $c \ind[K]_{M} a$, then there is $c' \equiv_{Ma} c$ such that $a \ind[K]_{M} bc'$ and $c' \ind[K]_{M} ab$.  
\end{cor}

\begin{proof}
By extension, choose $c_* \equiv_{M} c$ with $c_* \ind[K]_{M} b$. Then by Theorem~\ref{strongit}, there exists $c'$ such that $c'\equiv_{Ma} c$,  $c'\equiv_{Mb} c_*$, $c' \ind[K]_{M} ab$, $a \ind[K]_{M} bc'$, and $b \ind[K]_{M} ac'$, which is more than we need.
\end{proof}

\subsection{Kim-independence and algebraic independence}

We are now ready to show that Kim-independence satisfies the algebraically reasonable conditions of Definition~\ref{def:algreas} in any $\NSOP{1}$ theory.

\begin{thm}\label{thm:algreasext}
If $T$ is $\NSOP{1}$, then $\ind[K]$ satisfies algebraically reasonable extension.
\end{thm}
\begin{proof}

Suppose we have a model $M$ and tuples $a$, $b$, $c$, with $a\ind[K]_M b$. Let $\overline{b}$ be a tuple enumerating $\acl(Mb)$, and let $\overline{c} = (c_i)_{i\in I}$ be a tuple enumerating $\acl(Mbc)\setminus \acl(Mb)$. Then we also have $a\ind[K]_M \overline{b}$ (\cite[Corollary 5.17]{KRKim}). And for any $a'$, we have $a'\nind[a]_{M\overline{b}} \overline{c}$ if and only if there is some index $i\in I$ such that $c_i\in \acl(Ma'b)$.

Let $\kappa = |\acl(Mab)|$. Since $\overline{b}$ enumerates an algebraically closed set, we can find pairwise disjoint tuples $(\overline{c}^\alpha)_{\alpha< \kappa^+}$ such that $\overline{c}^\alpha \equiv_{M\overline{b}}\overline{c}$ for all $\alpha$. By extension, there exists $a''\equiv_{M\overline{b}} a$ such that $a''\ind[K]_M \overline{b}(\overline{c}^\alpha)_{\alpha< \kappa^+}$. In particular, for every $\alpha$, $a''\ind[K]_M \overline{b}\overline{c}^\alpha$. And since $|\acl(Ma''b)| = \kappa$, there is some $\alpha$ such that $\overline{c}^\alpha$ is disjoint from $\acl(Ma''b)$, so $a''\ind[a]_{M\overline{b}}\overline{c}^\alpha$.

Let $\sigma$ be an automorphism moving $\overline{c}^\alpha$ to $\overline{c}$ and fixing $\overline{b}$, and let $a' = \sigma(a'')$. Then $a'\ind[K]_M bc$ and $a'\ind[a]_{Mb}c$. 
\end{proof}

\begin{lem} \label{weirdextension}
Suppose $T$ is $\NSOP{1}$ and $M \models T$.  If $b \ind[K]_{M} a$ and $c \ind[K]_{M} a$, then there is $c' \equiv_{M a} c$ such that $bc' \ind[K]_{M} a$ and $b \ind[a]_{Ma} c'$.  
\end{lem}

\begin{proof}
We build a sequence $(b_i)_{i<\omega}$ by induction, such that for all $i<\omega$, the following conditions hold: 
\begin{enumerate}
\item $b_i \equiv_{Ma} b$.
\item $b_{i+1} \ind[K]_M ab_{\leq i}$.
\item $b_{i+1} \ind[a]_{Ma} b_{\leq i}$.
\end{enumerate}

Set $b_0 = b$, and given $b_{\leq i}$, as $b\ind[K]_M a$, use algebraically reasonable extension to find $b_{i+1}\equiv_{Ma} b$ such that $b_{i+1} \ind[K]_M a b_{\leq i}$ and $b_{i+1}\ind[a]_{Ma} b_{\leq i}$. 

In particular, note that $b_i\ind[a]_{Ma} b_j$ for all $i\neq j$.

Define a partial type $\Gamma(x;a,b)$ by 
$$
\Gamma(x;a,b) = \text{tp}(c/Ma) \cup \{\neg \varphi(x,b;a) \mid \varphi(x,y;a) \text{ Kim-divides over }M\}.
$$

\textbf{Claim}: For all $n<\omega$, $\bigcup_{i \leq n} \Gamma(x;a,b_{i})$ is consistent. 

\emph{Proof of claim}:  By induction on $n$, we will find $c_{n} \ind[K]_{M} ab_{\leq n}$ so that $c_{n} \models \bigcup_{i \leq n} \Gamma(x;a,b_{i})$, i.e.\ $c_n\equiv_{Ma} c$ and $b_ic\ind[K]_M a$ for all $i\leq n$.

For $n = 0$, the existence of such a $c_{0}$ is given by Corollary \ref{easycorollary}.  Suppose we have $c_{n} \ind[K]_{M} ab_{\leq n}$ realizing $\bigcup_{i \leq n} \Gamma(x;a,b_{i})$.  By extension, choose $c' \equiv_{M} c$ with $c' \ind[K]_{M} b_{n+1}$.  As $b_{n+1} \ind[K]_{M} ab_{\leq n}$, we may apply the strengthened independence theorem (Theorem \ref{strongit}), to find $c_{n+1} \models \text{tp}(c_{n}/Mab_{\leq n}) \cup \text{tp}(c'/Mb_{n+1})$ with $c_{n+1} \ind[K]_{M} ab_{\leq n+1}$ and $b_{n+1}c_{n+1} \ind[K]_{M} ab_{\leq n}$.  In particular, $b_{n+1}c_{n+1} \ind[K]_{M} a$, so $c_{n+1} \models \Gamma(x;a,b_{n+1})$.  This gives $c_{n+1} \models \bigcup_{i \leq n+1} \Gamma(x;a,b_{i})$. \qed

Let $\kappa = |\acl(Mac)|$, and let $(b_\alpha')_{\alpha<\kappa^+}$ be an $Ma$-indiscernible sequence locally based on $(b_i)_{i<\omega}$. Then we have $b'_\alpha\ind[a]_{Ma} b'_\beta$ for all $\alpha \neq \beta$, and $\bigcup_{\alpha<\kappa^+} \Gamma(x; a, b'_i)$ is consistent, by the claim. Let $c_*$ realize this partial type, so $c_*\equiv_{Ma} c$ and $b'_\alpha c_*\ind[K]_{M} a$ for all $\alpha$. 

Since the sets $\acl(Mab'_\alpha)$ are pairwise disjoint over $\acl(Ma)$, and $|\acl(Mac_*)| = \kappa$, there is some $\alpha<\kappa^+$ such that $\acl(Mab'_\alpha)$ is disjoint from $\acl(Mac)$ over $\acl(Ma)$, so $b'_\alpha\ind[a]_M c$. Since $b_\alpha\equiv_{Ma} b$, we can find an automorphism $\sigma$ fixing $Ma$ and moving $b'_\alpha$ to $b$. Taking $c' = \sigma(c_*)$, we have $c'\equiv_{Ma} c_*\equiv_{Ma} c$ and $bc'\ind[K]_M a$, as desired.   
\end{proof}

\begin{cor}\label{goodseq}
Suppose $T$ is $\NSOP{1}$ and $M \models T$.  If $b \ind[K]_{M} a$, then for any cardinal $\kappa$, there is an $Ma$-indiscernible sequence $I = (b_\alpha)_{\alpha < \kappa}$ with $b_{0} = b$, such that $b_\alpha \ind[a]_{Ma} b_\beta$ for all $\alpha \neq \beta$, and $I\ind[K]_M a$.
\end{cor}

\begin{proof}
We first build a sequence $(c_i)_{i < \omega}$ by induction, such that for all $i< \omega$, the following conditions hold:
\begin{enumerate}
\item $c_{i} \equiv_{Ma} b$.
\item $c_{i} \ind[a]_{Ma} c_{<i}$.
\item $c_{\leq i} \ind[K]_{M} a$.  
\end{enumerate}
Set $c_{0} = b$, and given $c_{\leq i}$, as $c_{\leq i} \ind[K]_{M} a$ and $b\ind[K]_M a$, we may apply Lemma \ref{weirdextension} to find $c_{i+1} \equiv_{Ma} b$ such that $c_{i+1} \ind[a]_{Ma} c_{\leq i}$ and $c_{\leq i+1} \ind[K]_{M} a$.

Now let $I = (b_\alpha)_{\alpha < \kappa}$ be an $Ma$-indiscernible sequence locally based on $(c_i)_{i<\omega}$. By condition (1), we may assume that $b_0 = b$. Condition (2) implies that $c_i \ind[a]_{Ma} c_j$ for all $i\neq j$, so $b_\alpha\ind[a]_{Ma} b_\beta$ for all $\alpha \neq \beta$, and $I\ind[K]_M a$ by condition (3) and the strong finite character of Kim-dividing.
\end{proof}

\begin{thm}\label{thm:algreaschain}
If $T$ is $\NSOP{1}$, then $\ind[K]$ satisfies the algebraically reasonable chain condition.
\end{thm}
\begin{proof}
Suppose $a\ind[K]_M b$, and let $I = (b_i)_{i\in \omega}$ be a Morley sequence over $M$ in a global $M$-invariant type $q\supseteq \tp(b/M)$. 

\textbf{Claim:} For all $n$, there exists $(c_0,\dots,c_n) \models q^{\otimes (n+1)}|_M$ such that $c_i\equiv_{Ma} b$ for all $i\leq n$, $c_i\ind[a]_{Ma} c_j$ for all $i\neq j$, and $(c_i)_{i\leq n} \ind[K]_M a$.

\emph{Proof of claim:} By induction on $n$. When $n= 0$, taking $c_0 = b$ suffices. So suppose we are given the tuple $(c_0,\dots,c_n)$ by induction. Let $\kappa = |\acl(Mab)|$, and, applying Corollary~\ref{goodseq}, let $J = (d_{0,\alpha})_{\alpha<\kappa^+}$ be an $Ma$-indiscernible sequence with $d_{0,0} = b$, such that $d_{0,\alpha}\ind[a]_{Ma} d_{0,\beta}$ for all $\alpha \neq \beta$, and $J \ind[K]_M a$.

Let $(d_1,\dots,d_{n+1})$ realize $q^{\otimes(n+1)}|_{MJ}$. Since $d_{0,\alpha}\models q|_M$ for all $\alpha$, we have $(d_{0,\alpha},d_1,\dots,d_{n+1})\models q^{\otimes(n+2)}|_M$ for all $\alpha$. Let $\sigma\in \Aut(\mathbb{M}/M)$ be such that $\sigma(c_i) = d_{i+1}$ for all $i\leq n$, and let $a' = \sigma(a)$. Now $a\equiv_M a'$, $a\ind[K]_M J$ (by choice of $J$), $a'\ind[K]_M (d_1,\dots,d_{n+1})$ (by invariance), and $(d_1,\dots,d_{n+1})\ind[K]_M J$ (since $\tp(d_1,\dots,d_{n+1}/MJ)$ extends to a global $M$-invariant type). 

Applying the independence theorem, we find $a''$ with $a''\equiv_{MJ} a$, $a''\equiv_{Md_1\dots d_{n+1}} a'$, and $a''\ind[K]_M Jd_1\dots d_{n+1}$. Then we still have $d_i \ind[a]_{Ma''} d_j$ for all $1\leq i < j \leq n+1$, and since the sets $\acl(Ma''d_{0,\alpha})$ are pairwise disjoint over $\acl(Ma'')$, and $|\acl(Ma''d_i)| = \kappa$ for all $1\leq i\leq n$, there is some $\alpha <\kappa^+$ such that $\acl(Ma''d_{0,\alpha})$ is disjoint from each of these $n$ sets over $\acl(Ma'')$. So setting $d_0 = d_{0,\alpha}$, we have $d_i \ind[a]_{Ma''} d_j$ for all $i\neq j$.

It remains to move $a''$ back to $a$ by an automorphism $\sigma\in \Aut(\mathbb{M}/M)$. The tuple $\sigma(d_0,\dots,d_{n+1})$ has the desired properties. \qed

By compactness, we can find $I' = (c_i)_{i<\omega}\models q^{\otimes\omega}$, i.e.\ a $q$-Morley sequence over $M$, such that $c_i\equiv_{Ma} b$ for all $i<\omega$, $c_i\ind[a]_{Ma} c_j$ for all $i\neq j$, and $I' \ind[K]_M a$. In fact, we can assume that $I'$ is $Ma$-indiscernible, by replacing it with an $Ma$-indiscernible sequence locally based on it. As $I'$ and $I$ are both $q$-Morley sequences over $M$, we can move $I'$ to $I$ by an automorphism $\sigma\in \Aut(\mathbb{M}/M)$, and take $a' = \sigma(a)$.
\end{proof}

\begin{lem}\label{lem:seq}
Suppose $(a_i)_{i<\omega}$ is a Morley sequence for a global $M$-invariant type, which is moreover $Mb$-indiscernible. If $b\ind[a]_{M(a_i)_{i<\omega}}b'$, then $b\ind[a]_{Ma_0}b'$. 
\end{lem}
\begin{proof}
Suppose there is some element $c\in \acl(Ma_0b)\cap \acl(Ma_0b')$. We would like to show that $c\in \acl(Ma_0)$. What we have is that $c\in \acl(M(a_i)_{i<\omega})$, and in particular $c\in \acl(Ma_0b) \cap \acl(Ma_0a_{i_1}\dots a_{i_n})$ for some $0 < i_1 < \dots < i_n$.  Now $(a_i)_{i\geq 1}$ is indiscernible over $Ma_0b$, hence indiscernible over $\acl(Ma_0b)$, which contains $c$. So $c$ is also in both $\acl(Ma_0a_1\dots a_n)$ and $\acl(Ma_0 a_{n+1}\dots a_{2n})$. 

But $\tp(a_{n+1},\dots,a_{2n}/Ma_0a_1\dots a_n)$ extends to a global $Ma_0$-invariant type, so we must have $a_{n+1}\dots a_{2n} \ind[a]_{Ma_0} a_1\dots a_n.$
Hence $c\in \acl(Ma_0)$.
\end{proof}

\begin{lem}\label{lem:mutind}
Given $a\ind[K]_M b$ and global $M$-invariant types $p(x)$ and $q(y)$ extending $\tp(a/M)$ and $\tp(b/M)$ respectively, there exist mutually indiscernible Morley sequences $(a_i)_{i<\omega}$ and $(b_i)_{i<\omega}$ in $p$ and $q$,  with $a_0 = a$ and $b_0 = b$, such that $(a_i)_{i<\omega}\ind[K]_M (b_i)_{i<\omega}$, and $a_i\ind[a]_{Mb} a_j$ and $b_i\ind[a]_{Ma} b_j$ for all $i\neq j$. 
\end{lem}
\begin{proof}
Let $(a_i)_{i<\omega}$ be a Morley sequence in $p$, with $a_0 = a$. By the algebraically reasonable chain condition, there is some $b'\equiv_{Ma} b$, such that $(a_i)_{i<\omega}$ is $Mb'$-indiscernible, $b'\ind[K]_{M} (a_i)_{i<\omega}$, and $a_i\ind[a]_{Mb'} a_j$ for all $i\neq j$. At the expense of moving $(a_i)_{i<\omega}$ by an automorphism fixing $Ma$, we may assume that $b' = b$. 

Now let $(b_i)_{i<\omega}$ be a Morley sequence in $q$, with $b_0 = b$. Since $(a_i)_{i<\omega}\ind[K]_M b$, we can find some $(a_i')_{i<\omega} \equiv_{Mb} (a_i)_{i<\omega}$, such that $(b_i)_{i<\omega}$ is $M(a_i')_{i<\omega}$-indiscernible, $(a_i')_{i<\omega} \ind[K]_M (b_i)_{i<\omega}$, and $b_i\ind[a]_{M(a_i')_{i<\omega}} b_j$ for all $i\neq j$. Further, we may replace $(a_i')_{i<\omega}$ with an $M(b_i)_{i<\omega}$-indiscernible sequence $(a_i'')_{i<\omega}$ locally based on it, and at the expense of moving $(b_i)_{i<\omega}$ by an automorphism fixing $Mb$, we may assume that $(a_i'')_{i<\omega} = (a_i)_{i<\omega}$.  

The result is that $(a_i)_{i<\omega}$ and $(b_i)_{i<\omega}$ are mutually indiscernible Morley sequences in $p$ and $q$ with $a_0 = a$ and $b_0 = b$ and $(a_i)_{i<\omega}\ind[K]_M (b_i)_{i<\omega}$. We have also ensured that $a_i\ind[a]_{Mb} a_j$ for all $i\neq j$ and $b_i\ind[a]_{M(a_i)_{i<\omega}} b_j$ for all $i\neq j$. By Lemma~\ref{lem:seq}, also $b_i\ind[a]_{Ma} b_j$ for all $i\neq j$. 
\end{proof}

\begin{thm}\label{thm:algreasind}
If $T$ is $\NSOP{1}$, then $\ind[K]$ satisfies the algebraically reasonable independence theorem.
\end{thm}
\begin{proof}
We have a model $M$ and tuples $a,a',b,c$, with $a\ind[K]_M b$, $a'\ind[K]_M c$, $b\ind[K]_M c$, and $a\equiv_M a'$. Let $p(x)$, $q(y)$, and $r(z)$ be global $M$-invariant types extending $\tp(a/M) = \tp(a'/M)$, $\tp(b/M)$, and $\tp(c/M)$, respectively.

Apply Lemma~\ref{lem:mutind} to $q(y)$ and $r(z)$, obtaining Morley sequences $(b_i)_{i<\omega}$ and $(c_i)_{i<\omega}$. Then apply it two more times, to $p(x)$ and $q(y)$, obtaining Morley sequences $(a_i)_{i<\omega}$ and $(\widehat{b_i})_{i<\omega}$, and then to $p(x)$ and $r(z)$, obtaining Morley sequences $(a_i')_{i<\omega}$ and $(\widehat{c_i})_{i<\omega}$. At the expense of moving $(a_i)_{i<\omega}$ and $(a_i')_{i<\omega}$ by automorphisms over $M$, we may assume that $(\widehat{b_i})_{i<\omega} = (b_i)_{i<\omega}$ and $(\widehat{c_i})_{i<\omega} = (c_i)_{i<\omega}$. Note that $(a_i)_{i<\omega}$ and $(a_i')_{i<\omega}$ are both $p$-Morley sequences over $M$, so $(a_i)_{i<\omega}\equiv_M (a_i')_{i<\omega}$.  

We now apply the independence theorem to the sequences $(a_i)_{i<\omega}$, $(a_i')_{i<\omega}$, $(b_i)_{i<\omega}$, and $(c_i)_{i<\omega}$, obtaining a sequence $(a_i'')_{i<\omega}$ such that $(a_i'')_{i<\omega}\equiv_{M(b_i)_{i<\omega}} (a_i)_{i<\omega}$, $(a_i'')_{i<\omega}\equiv_{M(c_i)_{i<\omega}} (a_i')_{i<\omega}$, and $(a_i'')_{i<\omega} \ind[K]_M (b_i)_{i<\omega}(c_i)_{i<\omega}$. The sequences $(a_i'')_{i<\omega}$, $(b_i)_{i<\omega}$, and $(c_i)_{i<\omega}$ are pairwise mutually indiscernible over $M$ and have the property that any pair from one sequence is algebraically independent over any element of another sequence. 

Let $\kappa$ be a cardinal larger than the sizes of $M$, the language, and the tuples $a$, $b$, and $c$. We can stretch the $(b_i)$ sequence to have length $\kappa^+$ and stretch the $(c_i)$ sequence to have length $\kappa^{++}$, while maintaining their mutual indiscernibility and algebraic independence properties. 

Fix $a_* = a''_0$. For any $i<\kappa^+$, $|\acl(Ma_*b_i)| \leq \kappa$. Since the sets $\{\acl(Ma_*c_j)\mid j<\kappa^{++}\}$ are pairwise disjoint over $\acl(Ma_*)$, we can remove any $c_j$ such that there exists an $i<\kappa^+$ such that $b_i\nind[a]_{Ma_*} c_j$, and we are still left with a sequence of length $\kappa^{++}$. 

Similarly, since for all $i<\kappa^+$, the sets $\{\acl(Mb_ic_j)\mid j<\kappa^{++}\}$ are pairwise disjoint over $\acl(Mb_i)$, we can further remove any $c_j$ such that there exists an $i$ such that $a\nind[a]_{Mb_i} c_j$, and we are still left with a sequence of length $\kappa^{++}$. Fix a $c_*$ from this sequence.

Finally, since the sets $\{\acl(Mb_ic_*)\mid i<\kappa^+\}$ are pairwise disjoint over $\acl(Mc_*)$, we can remove any $b_i$ such that $a_*\nind[a]_{Mc_*}b_i$, and we are still left with a sequence of length $\kappa^+$. Fix a $b_*$ from this sequence.

It remains to move $b_*c_*$ back to $bc$ by an automorphism $\sigma$ fixing $M$, and set $a'' = \sigma(a_*)$.
\end{proof}

\section{The model companion of the empty theory}\label{sec:emptytheory}

\subsection{The theory $T^\emptyset_L$}

Let $L$ be any language. Then the empty $L$-theory has a model completion, which we call $T^\emptyset_L$. As the theory $T^{\emptyset}_{L}$ may be regarded as the generic expansion of the theory of an infinite set in the empty language, this fact is a special case of Theorem 5 in~\cite{winkler1975model} (see Fact~\ref{winkler} below), and it was reproven by Je\v{r}\'{a}bek in~\cite{jerabek}. We include a proof, following the idea of \cite{winkler1975model}, for completeness and to fix notation.

\begin{defn}
A \emph{partial diagram} $\Delta$ is a set of atomic and negated atomic formulas. $\Delta$ is \emph{flat} if each formula in $\Delta$ has the form $R(\overline{z})$, $\lnot R(\overline{z})$, or $f(\overline{z}) = z'$, where $R$ is a relation symbol, $f$ is a function symbol, $\overline{z}$ is a tuple of variables and $z'$ is a single variable. We view constant symbols as $0$-ary function symbols, so this includes formulas of the form $c = z'$. 
\end{defn}

In a flat diagram, we always intend distinct variables to refer to distinct elements.

\begin{defn}
A flat diagram $\Delta$ is \emph{consistent} if, for each tuple of variables $\overline{z}$, 
\begin{enumerate}
\item At most one of $R(\overline{z})$ and $\lnot R(\overline{z})$ is in $\Delta$, where $R$ is a relation symbol.
\item There is at most one variable $z'$ such that $f(\overline{z}) = z'$ is in $\Delta$, where $f$ is a function symbol.  
\end{enumerate}
\end{defn}

\begin{defn}
A consistent flat diagram $\Delta$ in the variables $\overline{w}$ is \emph{complete} if, for each relation symbol $R$, each function symbol $f$, and each tuple of variables $\overline{z}$ from $\overline{w}$ of the appropriate length,
\begin{enumerate}
\item Either $R(\overline{z})$ or $\lnot R(\overline{z})$ is in $\Delta$. 
\item There is some variable $z'$ in $\overline{w}$ such that $f(\overline{z}) = z'$ is in $\Delta$.
\end{enumerate}
\end{defn}

Let $A$ be any $L$-structure. Then there is a complete flat diagram $\diagf(A)$ in the variables $(w_a)_{a\in A}$, which contains a formula $\psi(w_{a_1},\dots,w_{a_n})$ of one of the allowed forms if and only if $A\models \psi(a_1,\dots,a_n)$. The following easy lemma establishes the converse.

\begin{lem}\label{lem:diagram}
Suppose $\Delta$ is a consistent flat diagram in the nonempty set of variables $(w_a)_{a\in A}$. Then there is an $L$-structure with domain $A$ such that for all $\psi(w_{a_1},\dots,w_{a_n})\in \Delta$, $A\models \psi(a_1\dots,a_n)$. 
\end{lem}
\begin{proof}
First, we extend $\Delta$ to a complete flat diagram $\Delta'$ as follows: For each $n$-ary relation symbol $R$, and for each $n$-tuple $\overline{z}$ such that neither $R(\overline{z})$ nor $\lnot R(\overline{z})$ is in $\Delta$, add $\lnot R(\overline{z})$ to $\Delta'$. Now fix an arbitrary variable $w_a$. For each $n$-ary function symbol $f$, and for each $n$-tuple $\overline{z}$ such that no formula of the form $f(\overline{z}) = z'$ is in $\Delta$, add $f(\overline{z}) = w_a$ to $\Delta'$. 

We define an $L$-structure with domain $A$, according to $\Delta'$. If $R$ is an $n$-ary relation symbol, we set $R^A = \{(a_1,\dots,a_n)\in A^n\mid R(w_{a_1},\dots,w_{a_n})\in \Delta'\}$.
If $f$ is an $n$-ary function symbol and $(a_1,\dots,a_n)\in A^n$, we set $f^A(a_1,\dots,a_n) = a'$, where $a'$ is the unique element of $A$ such that $f(w_{a_1},\dots,w_{a_n}) = w_{a'}\in \Delta'$. Consistency ensures that this $L$-structure is well-defined and satisfies all the formulas in $\Delta'$ (and hence in $\Delta$).
\end{proof}

For the purposes of axiomatizing the existentially closed $L$-structures, we will be interested in a class of finite partial diagrams, which we call extension diagrams.

\begin{defn}
Let $\overline{w}$ be a finite tuple of variables, partitioned into two subtuples $\overline{x}$ and $\overline{y}$. An \emph{extension diagram} in $(\overline{x},\overline{y})$ is a consistent flat diagram $\Delta$ in the variables $\overline{w}$, such that for each formula $R(\overline{z})$, $\lnot R(\overline{z})$, or $f(\overline{z}) = z'$ in $\Delta$, some variable in $\overline{z}$ is in $\overline{y}$. In particular, no constant symbols appear in extension diagrams.
\end{defn}

A tuple $\overline{a} = (a_i)_{i\in I}$ is \emph{non-redundant} if $a_i\neq a_j$ for all $i\neq j$. Given a finite tuple of variables $\overline{z} = (z_1,\dots,z_n)$, let $\delta(\overline{z})$ be the formula which says that $\overline{z}$ is non-redundant: $$\bigwedge_{1\leq i<j\leq n} z_i\neq z_j.$$ 

Given a finite partial diagram $\Delta$ in the finite tuple of variables $\overline{w}$, let $\varphi_{\Delta}(\overline{w})$ be the conjunction of all the formulas in $\Delta$, together with $\delta(\overline{w})$: $$ \left(\bigwedge_{\psi(\overline{z})\in \Delta}\psi(\overline{z})\right)\land \delta(\overline{w}).$$

\begin{lem}\label{lem:extension}
Let $\Delta$ be an extension diagram in $(\overline{x},\overline{y})$, and let $A$ be an $L$-structure. If $\overline{a}$ is a non-redundant tuple from $A$ of the same length as $\overline{x}$, then there is an $L$-structure $B$ containing $A$ and a tuple $\overline{b}$ from $B$ of the same length as $\overline{y}$ such that $B\models \varphi_\Delta(\overline{a},\overline{b})$. 
\end{lem}
\begin{proof}
Consider the flat diagram $\diagf(A)\cup \Delta(x_{a_1},\dots,x_{a_n},y_{b_1},\dots,y_{b_n})$, where we identify the variables $\overline{x}$ in $\Delta$ with the variables in $\diagf(A)$ enumerating $\overline{a}$, and we index the variables $\overline{y}$ in $\Delta$ by a new tuple $\overline{b}$. This diagram is consistent, since $\diag_f(A)$ and $\Delta$ are individually consistent, and for every formula $R(\overline{z})$, $\lnot R(\overline{z})$, or $f(\overline{z}) = z'$ in $\Delta$, some element of the tuple $\overline{z}$ is in $\overline{y}$, while $\diagf(A)$ does not mention the variables in $\overline{y}$. Hence, by Lemma~\ref{lem:diagram}, there is a structure $B$ with domain $A\cup \{b_1,\dots,b_n\}$, such that $B$ satisfies $\diagf(A)$ (so $A$ is a substructure of $B$), and $B\models \varphi_\Delta(\overline{a},\overline{b})$.
\end{proof}

\begin{lem}\label{lem:flattening}
If an $L$-structure $A$ is not existentially closed, then there is a non-redundant tuple $\overline{a}$ from $A$ and an extension diagram $\Delta$ in $(\overline{x},\overline{y})$, such that $A\models \lnot \exists \overline{y}\,\varphi_\Delta(\overline{a}, \overline{y})$.
\end{lem}

\begin{proof}
Since $A$ is not existentially closed, there is a quantifier-free $L$-formula $\varphi(\overline{x},\overline{y})$, an $L$-structure $B$ containing $A$, and tuples $\overline{a}\in A$ and $\overline{b}\in B$, such that $B\models \varphi(\overline{a},\overline{b})$, but $A\models \lnot\exists \overline{y}\, \varphi(\overline{a},\overline{y})$. We may assume the the tuples $\overline{a}$ and $\overline{b}$ are non-redundant and that $b_i\in B\setminus A$ for all $i$. Writing $\varphi$ in disjunctive normal form, one of the disjuncts is satisfied by $(\overline{a},\overline{b})$ in $B$, so we may assume that $\varphi$ is a conjunction of atomic and negated atomic formulas. Let $\Delta$ be the finite partial diagram containing these formulas. Then $\varphi_\Delta(\overline{x},\overline{y})$ is equivalent to $\varphi(\overline{x},\overline{y})\land \delta(\overline{x},\overline{y})$, and we have  $B\models\varphi_\Delta(\overline{a},\overline{b})$, but $A\models \lnot\exists \overline{y}\,\varphi_{\Delta}(\overline{a},\overline{y})$. 

We will transform $\Delta$ into an extension diagram. This process will involve adding and deleting variables and making corresponding changes to the tuples $\overline{a}$ and $\overline{b}$, but we will maintain the invariants that $\Delta$ is finite, $A\models \lnot\exists \overline{y}\,\varphi_{\Delta}(\overline{a},\overline{y})$, $B\models\varphi_{\Delta}(\overline{a},\overline{b})$, and $b_i\in B\setminus A$ for all $i$. We write $\overline{w}$ for the tuple $(\overline{x},\overline{y})$ and $\overline{c}$ for $(\overline{a},\overline{b})$. 

First, we flatten $\Delta$. Suppose that there is an $n$-ary function symbol $f$ such that the term $f(w_{i_1},\dots,w_{i_k})$ (where the $w_{i_j}$ are variables) appears in a formula in $\Delta$ which is not of the form $f(w_{i_1},\dots,w_{i_k}) = w'$ for some variable $w'$. Let $d = f^B(c_{i_1},\dots,c_{i_k})$. If $d = c_{i_{k+1}}$ for some $i_{k+1}$, then we simply replace this instance of $f(w_{i_1},\dots,w_{i_k})$ with $w_{i_{k+1}}$ and add the formula $f(w_{i_1},\dots,w_{i_k}) = w_{i_{k+1}}$ to $\Delta$ if it is not already there. If $d$ is not in the tuple $\overline{c}$, we introduce a new variable $w'$ (a new $x$ if $d\in A$ and a new $y$ otherwise), add $d$ to $\overline{c}$ (as a new $a$ if $d\in A$ and a new $b$ otherwise), replace this instance of $f(w_{i_1},\dots,w_{i_k})$ with $w'$, and add the formula $f(w_{i_1},\dots,w_{i_k}) = w'$ to $\Delta$. 

Repeating this procedure, we eventually ensure that every formula in $\Delta$ has the form $w = w'$, $w\neq w'$, $R(w_{i_1},\dots,w_{i_n})$, $\lnot R(w_{i_1},\dots,w_{i_n})$, or $f(w_{i_1},\dots,w_{i_n}) = w'$.

Next we remove the equations and inequations between variables. Since the tuples $\overline{a}$ and $\overline{b}$ are non-redundant, $\Delta$ does not contain any equalities  between distinct variables, and the equalities of the form $w = w$ can of course be removed. Further, we may assume that $\Delta$ does not contain any inequalities $w_i\neq w_j$ between variables either, since these inequalities are all implied by $\delta(\overline{x},\overline{y})$ and hence by $\varphi_\Delta$. The set of formulas $\Delta$ is now a flat diagram. It is consistent, since it is satisfied by the non-redundant tuple $\overline{c}$.

Finally, let $\Delta'$ be the extension diagram obtained by removing from $\Delta$ any formula $R(\overline{z})$, $\lnot R(\overline{z})$, or $f(\overline{z}) = z'$ in which none of the variables in $\overline{z}$ are in $\overline{y}$. Note that in the case of $f(\overline{z}) = z'$, if all of the $\overline{z}$ are in $\overline{x}$, then their interpretations come from $A$, and since $A$ is closed under the function symbols, $z'$ is in $\overline{x}$ as well. 

So $\varphi_{\Delta}(\overline{x},\overline{y})$ is equivalent to $\varphi_{\Delta'}(\overline{x},\overline{y})\land \bigwedge_{j = 1}^N\psi_j(\overline{x})$, where each $\psi_j$ is atomic or negated atomic. But since $$A\models \lnot \exists \overline{y}\, \left(\varphi_{\Delta'}(\overline{a},\overline{y})\land \bigwedge_{j = 1}^N\psi_j(\overline{a})\right),$$ also $A\models \lnot \exists \overline{y}\, \varphi_{\Delta'}(\overline{a},\overline{y})$, as was to be shown.
\end{proof}

Given an extension diagram $\Delta$ in $(\overline{x},\overline{y})$, let $\psi_\Delta$ be the sentence $$\forall \overline{x}\,(\delta(\overline{x})\rightarrow \exists\overline{y}\,\varphi_{\Delta}(\overline{x},\overline{y})),$$ and let $T^\emptyset_L = \{\psi_\Delta\mid \Delta\text{ is an extension diagram in }(\overline{x},\overline{y})\}$.

\begin{thm}
$T^\emptyset_L$ is the model companion of the empty $L$-theory.
\end{thm}
\begin{proof}
It suffices to show that the class of existentially closed $L$-structures is axiomatized by $T^\emptyset_L$~\cite[Proposition 3.5.15]{ChangKeisler}.

Suppose $A$ is an existentially closed $L$-structure, and let $\Delta$ be an extension diagram in $(\overline{x},\overline{y})$. Let $\overline{a}$ be any non-redundant tuple from $A$ of the same length as $\overline{x}$. By Lemma~\ref{lem:extension}, there is an $L$-structure $B$ containing $A$ and a tuple $\overline{b}$ from $B$ such that $B\models \varphi_\Delta(\overline{a},\overline{b})$. So $B\models \exists\overline{y}\,\varphi_\Delta(\overline{a},\overline{y})$. But since $A$ is existentially closed, also $A\models \exists\overline{y}\,\varphi_\Delta(\overline{a},\overline{y})$. Hence $A\models \psi_\Delta$, and since $\Delta$ was arbitrary, $A\models T^\emptyset_L$.

Conversely, suppose the $L$-structure $A$ is not existentially closed. By Lemma~\ref{lem:flattening}, there is a non-redundant tuple $\overline{a}$ from $A$ and an extension diagram $\Delta$ in $(\overline{x},\overline{y})$ such that $A\models \lnot\exists\overline{y}\,\varphi_\Delta(\overline{a},\overline{y})$. Hence $A\not\models \psi_\Delta$, and $A\not\models T^\emptyset_L$.
\end{proof}

\begin{lem}\label{lem:disjointamalgamation}
The class of $L$-structures satisfies the disjoint amalgamation property.
\end{lem}
\begin{proof}
Let $f_1\colon A\to B$ and $f_2\colon A\to C$ be embeddings of $L$-structures. Let $B' = B\setminus f_1(A)$ and $C' = C\setminus f_2(A)$, and consider the diagrams $\diagf(B) = \Delta_B((x_a)_{a\in A},(x_b)_{b\in B'})$ and  $\diagf(C) = \Delta_C((x_a)_{a\in A},(x_c)_{c\in C'})$, where we use the same variables $(x_a)_{a\in A}$ to enumerate $f_1(A)$ and $f_2(A)$. 

Then $\Delta_B\cup \Delta_C$ is consistent, since the two diagrams agree on $\diagf(A)$. By Lemma~\ref{lem:diagram}, we get an $L$-structure $D$ with domain $A\cup B' \cup C'$, and the obvious maps $g_1\colon B\to D$ and $g_2\colon C\to D$ satisfy $g_1\circ f_1 = g_2\circ f_2$. These maps are embeddings, since $D$ satisfies $\Delta_B$ and $\Delta_C$, and the images of $B$ and $C$ are disjoint over the image of $A$.
\end{proof}

The following corollary now follows from standard facts about model completions (see \cite[Proposition 3.5.18]{ChangKeisler}).

\begin{cor}\label{cor:qe}
$T^\emptyset_L$ is a model completion of the empty $L$-theory, and it has quantifier elimination. The completions of $T^\emptyset_L$ are obtained by specifying (by quantifier-free sentences) the isomorphism type of the structure $\langle \emptyset\rangle$ generated by the constants. Such a completion $\widetilde{T}$ is the model completion of the theory of $L$-structures containing a substructure isomorphic to $\langle\emptyset\rangle$. If there are no constant symbols in $L$, then $T^\emptyset_L$ is complete.
\end{cor}

\begin{cor}
Let $\mathbb{M}$ be a monster model for some completion of $T^\emptyset_L$. For any set $A\subseteq \mathbb{M}$, $\acl(A) = \dcl(A) = \langle A\rangle$. 
\end{cor}

\begin{proof}
Since $\langle A\rangle \subseteq \dcl(A) \subseteq \acl(A)$, it suffices to show that $\acl(A)\subseteq \langle A\rangle$. Suppose $\varphi(\overline{a},y)$ is an algebraic formula with parameters $\overline{a}$ from $A$, which is satisfied by exactly $k$ elements of $\mathbb{M}$, including $b$. By Corollary~\ref{cor:qe}, we may assume that $\varphi$ is quantifier-free. Suppose for contradiction that $b\notin \langle A\rangle$, so that $\langle A\rangle$ is a proper substructure of $B = \langle Ab\rangle$.  Let $C_{0} = B$ and, by induction, apply Lemma~\ref{lem:disjointamalgamation} to obtain a disjoint amalgam $C_{i+1}$ of $C_{i}$ and $B$ over $\langle A \rangle$.  Let $B_{i}$ denote the image of $B$ in $C_{i}$.  Then $C_{k+1}$ contains $\langle A\rangle$, together with substructures $B_1,\dots,B_{k+1}$, pairwise disjoint over $\langle A\rangle$ and each isomorphic to $B$ over $\langle A\rangle$. Then, by quantifier elimination and saturation, $C_{k+1}$ embeds in $\mathbb{M}$ over $\langle A\rangle$, and we may identify the $B_i$ with their images in $\mathbb{M}$. Each $B_i$ contains an element $b_i$ such that $\qftp(b_i/A) = \qftp(b/A)$, so $\mathbb{M}\models \varphi(\overline{a},b_i)$ for all $i$, which is a contradiction.
\end{proof}

\subsection{Independence and $\NSOP{1}$}

For the remainder of this section, we fix a monster model $\mathbb{M}\models T^\emptyset_L$.  As there is a monster model for every choice of completion of $T^\emptyset_L$ and $\mathbb{M}$ is arbitrary, to show that $T^\emptyset_L$ is $\NSOP{1}$, it suffices to establish this for $\text{Th}(\mathbb{M})$.  

\begin{thm}\label{thm:algind}
$\ind[a]$ satisfies the independence theorem over arbitrary sets.
\end{thm}
\begin{proof}
Suppose we are given $C \subseteq \mathbb{M}$ and tuples $a,a',b,c$, with $a\ind[a]_C b$, $a'\ind[a]_C c$, $b\ind[a]_C c$, and $a\equiv_C a'$. Let $x_C$ be a tuple enumerating $\langle C\rangle$, let $x_a$, $x_b$ and $x_c$ be tuples enumerating $\langle Ca\rangle\setminus \langle C\rangle$, $\langle Cb\rangle\setminus \langle C\rangle$, and $\langle Cc\rangle\setminus \langle C\rangle$, respectively, and let $x_{ab}$, $x_{ac}$, and $x_{bc}$ be tuples enumerating $\langle Cab\rangle\setminus(\langle Ca\rangle\cup \langle Cb\rangle)$, $\langle Ca'c\rangle\setminus(\langle Ca'\rangle\cup \langle Cc\rangle)$, and $\langle Cbc\rangle\setminus(\langle Cb\rangle\cup \langle Mc\rangle)$, respectively. 

Observe that $(x_C,x_a,x_b,x_{ab})$ enumerates $\langle Cab\rangle$ without repetitions. The only thing to check is that no elements of $x_a$ and $x_b$ name the same element of $\langle Cab\rangle$, and this is exactly the condition that $a\ind[a]_C b$. Similarly, $(x_C,x_a,x_c,x_{ac})$ enumerates $\langle Ca'c\rangle$ (where we view $x_a$ as enumerating $\langle Ca'\rangle\setminus \langle C\rangle$ via the isomorphism $\langle Ca\rangle \to \langle Ca'\rangle$ induced by $a\mapsto a'$), and $(x_C,x_b,x_c,x_{bc})$ enumerates $\langle Cbc\rangle$.

Let $p_{ab} = \diagf(\langle Cab\rangle)$, $p_{ac} = \diagf(\langle Ca'c\rangle)$, and $p_{bc} = \diagf(\langle Cbc\rangle)$. The flat diagram $p_{ab}(x_C,x_a,x_b,x_{ab})\cup p_{ac}(x_C,x_a,x_c,x_{ac})\cup p_{bc}(x_C,x_b,x_c,x_{bc})$ is consistent, since $p_{ab}$, $p_{ac}$, and $p_{bc}$ agree on $\diagf(\langle Ca\rangle) = \diagf(\langle Ca'\rangle)$ (again, allowing $x_a$ to enumerate $\langle Ca'\rangle\setminus \langle C\rangle$), $\diagf(\langle Cb\rangle)$, and $\diagf(\langle Cc\rangle)$. So by Lemma~\ref{lem:diagram}, it extends to the flat diagram of an $L$-structure $X$ with domain $x_C\cup x_a\cup x_b\cup x_c\cup x_{ab}\cup x_{ac}\cup x_{bc}$.

Having constructed $X$, which agrees with $\mathbb{M}$ on the substructure generated by the empty set, we can embed it in $\mathbb{M}$ by $i\colon X\to \mathbb{M}$. Further, $$\qftp(i(x_C),i(x_b),i(x_c),i(x_{bc})) = \qftp(\langle Cbc\rangle),$$ so by quantifier elimination $$(i(x_C),i(x_b),i(x_c),i(x_{bc})) \equiv \langle Cbc\rangle,$$ and, by an automorphism of $\mathbb{M}$, we may assume that $i(x_C,x_b,x_c,x_{bc}) = \langle Cbc\rangle$. Let $a''$ the subtuple of $i(x_C,x_a)$ corresponding to the subtuple of $(x_C,x_a)$ enumerating $a$.

Now $\qftp(i(x_C),i(x_a),i(x_b),i(x_{ab})) = \qftp(\langle Cab\rangle)$, so $a''\equiv_{Cb} a$, and similarly $\qftp(i(x_C),i(x_a),i(x_c),i(x_{ac})) = \qftp(\langle Cac\rangle)$, so $a''\equiv_{Cc} a'$. Finally, $a''\ind[a]_C bc$, since $i(x_a)$ enumerates $\langle Ca''\rangle\setminus C$ and is disjoint from $(i(x_b),i(x_c),i(x_{bc}))$, which enumerates $\langle Cbc\rangle\setminus C$.
\end{proof}

\begin{cor}\label{cor:characterization}
$T^\emptyset_L$ is $\NSOP{1}$ and $\ind[K] = \ind[a]$ over models.
\end{cor}
\begin{proof}
By Theorem~\ref{thm:algind}, Lemma~\ref{algprops}, and Theorem~\ref{criterion}.
\end{proof}

On the other hand, except in trivial cases, $T^\emptyset_L$ has TP$_{2}$ and therefore is not simple.  For definitions of simple and TP$_{2}$ see, e.g., \cite{ChernikovNTP2}.  

\begin{prop}\label{prop:TP2}
If $L$ contains at least one $n$-ary function symbol with $n\geq 2$, then $T^\emptyset_L$ has TP$_2$, and is therefore not simple.  
\end{prop}
\begin{proof}
In $\mathbb{M} \models T^\emptyset_L$, choose a set of pairwise distinct $(n-1)$-tuples $B = \{b_{i} : i < \omega\}$ and a set of pairwise distinct elements $C = \{c_{i,j} : i, j < \omega\}$ so that $B$ and $C$ are disjoint.  Note that:
\begin{itemize}
\item For all $i < \omega$, $\{f(x,b_{i}) = c_{i,j} \mid j < \omega\}$ is 2-inconsistent.
\item For all $g: \omega \to \omega$, $\{f(x,b_{i}) = c_{i,g(i)} \mid i < \omega\}$ is consistent.
\end{itemize}
Hence the formula $\varphi(x;y,z)$ given by $f(x,y) = z$ has TP$_{2}$, witnessed by the array $(b_i,c_{i,j})_{i<\omega,j<\omega}$.
\end{proof}

\subsection{Forking and dividing}

Next, we analyze forking and dividing in $T^\emptyset_L$. See~\cite[Section 5]{Adler} for the definitions of forking and dividing. We begin with an example of the distinction between forking independence and Kim-independence.

\begin{exmp}\label{ex:binaryfct}
Suppose $L$ contains an $n$-ary function symbol $f$ with $n\geq 2$. Let $M \models T^\emptyset_L$, let $b$ be any $(n-1)$-tuple not in $M$, let $c$ be any element not in $\langle Mb\rangle$, and let $a$ be an element satisfying $a\ind[a]_M bc$, but $f(a,b) = c$.  By Corollary~\ref{cor:characterization}, $a \ind[K]_{M} bc$.  But $c\in (\langle Mab\rangle\cap \langle Mbc\rangle)\setminus \langle Mb\rangle$, so $a \nind[a]_{Mb} c$. Then also $a\nind[f]_{M} bc$, since otherwise we would have $a\ind[f]_{Mb} c$ and hence $a\ind[a]_{Mb} c$, by base monotonicity.

This example is closely related to the TP$_2$ array in Proposition~\ref{prop:TP2}. The formula $f(x,b) = c$ divides along any sequence $(bc_i)_{i<\omega}$ where $b$ is constant but the $c_i$ are distinct (like the rows of the TP$_2$ array). But this formula does not Kim-divide, since in any Morley sequence $(b_ic_i)_{i<\omega}$ for a global $M$-invariant type extending $\tp(bc/M)$, the $b_i$ are distinct (like the columns of the TP$_2$-array). See also Proposition~\ref{prop:forkdiv} below.
\end{exmp}
 
\begin{defn}
For subsets $A$, $B$, and $C$ of $\mathbb{M}$, we define $$A\ind[M]_C B \iff \text{for all }C\subseteq C'\subseteq \acl(BC), A\ind[a]_{C'} B.$$
\end{defn}

We will show that $\ind[M]$ agrees with dividing independence $\ind[d]$ in $T^\emptyset_L$ (Proposition~\ref{dividing}). The notation $\ind[M]$ comes from Adler~\cite{Adler}, who calls this relation ``$M$-dividing independence''. Adler shows that algebraic independence $\ind[a]$ satisfies all of his axioms for a strict independence relation except possibly base monotonicity (which it fails in $T^\emptyset_L$ whenever there is an $n$-ary function symbol, $n\geq 2$). 

The relation $\ind[M]$ is obtained from $\ind[a]$ by forcing base monotonicity, and it satisfies all of the axioms of a strict independence relation except possibly local character and extension. If we go one step further and force extension, we get the relation $\ind[\text{\th}]$ of thorn forking independence~\cite[Section 4]{Adler}, just as we get the relation $\ind[f]$ of forking independence by forcing extension on $\ind[d]$. But, as we will see, $\ind[M]$ already satisfies extension in $T^\emptyset_L$ (Proposition~\ref{extension}), so $M$-dividing independence, thorn forking independence, dividing independence, and forking independence all coincide in $T^\emptyset_L$. Of course, when $L$ contains an $n$-ary function symbol with $n\geq 2$, these independence relations lack local character, since $T^\emptyset_L$ is not simple, so $T^\emptyset_L$ is also not rosy. In contrast, $\ind[a]=\ind[K]$ has local character but lacks base monotonicity. This tension between local character and base monotonicity is characteristic of the difference between forking independence and Kim-independence in $\NSOP{1}$ theories.

\begin{prop}\label{dividing}
In $T^\emptyset_L$, $\ind[d] = \ind[M]$.
\end{prop}
\begin{proof}
In any theory, if $A\ind[d]_C B$, then $A\ind[M]_C B$~\cite[Remark 5.4.(4)]{Adler}. So suppose $A\ind[M]_C B$. We may assume $B = \acl(BC) = \langle BC\rangle$, since $A\ind[M]_C B$ implies $A\ind[M]_C \acl(BC)$ and $A\ind[d]_C \acl(BC)$ implies $A\ind[d]_C B$. 

Let $b$ be a tuple enumerating $B$, and let $(b_i)_{i<\omega}$ be a $C$-indiscernible sequence, with $b_0 = b$. Let $B_i$ be the set enumerated by $b_i$. Let $C'$ be the set of all elements of $B$ which appear in some $b_i$ for $i\neq 0$. Then $C\subseteq C'\subseteq B$, every element of $C'$ appears in every $b_i$, and $C' = \langle C'\rangle$. Letting $c'$ enumerate $C'$ and writing $b_i = (c',b_i')$ for all $i$, we have that $(b_i')_{i<\omega}$ is a $C'$-indiscernible sequence, and $b_i'$ and $b_j'$ are disjoint for all $i\neq j$. 

Let the tuple $x$ enumerate $\langle AC'\rangle \setminus C'$, and let the tuple $y_0$ enumerate $\langle AB\rangle \setminus (\langle AC'\rangle \cup B)$. By assumption, we have $A\ind[a]_{C'} B$, so $(x,y_0,c',b_0')$ enumerates $\langle AB\rangle$ without repetitions (since no elements of $x$ and $b_0'$ are equal). Let $D = \langle \bigcup_{i<\omega} B_i\rangle$, and let $d$ enumerate $D\setminus \bigcup_{i<\omega} B_i$. 

Let $p(x,y_0,c',b_0') = \diag_f(\langle AB\rangle)$, and let $q(c',(b_i')_{i<\omega},d) = \diag_f(D)$. Consider the flat diagram $q(c',(b_i')_{i<\omega},d)\cup \bigcup_{i<\omega} p(x,y_i,c',b_i')$, where the $y_i$ for $i>0$ are new tuples. This is consistent, since any two copies of $p$ agree on $\diag_f(\langle AC'\rangle)$, and the copy of $p$ indexed by $i$ agrees with $q$ on $\diag_f(B_i)$. So by Lemma~\ref{lem:diagram}, there is an $L$-structure $X$ with this diagram, and we can embed $X$ in $\mathbb{M}$ over $D$ by $i\colon X\to \mathbb{M}$, since $\qftp_X(c',(b_i')_{i<\omega},d') = \qftp_{\mathbb{M}}(c',(b_i')_{i<\omega},d')$.

Letting $A'$ be the subset of $(i(x),c')$ corresponding to $A$ as a subset of $(x,c')$, we have by quantifier elimination that $\tp(A'/B_i) = \tp(A/B)$ for all $i$. So $A\ind[d]_C B$. 
\end{proof}

\begin{prop}\label{extension}
The relations $\ind[M]$ and $\ind[d]$ satisfy extension over arbitrary sets, so $\ind[M] = \ind[\text{\th}]$ and $\ind[d] = \ind[f]$.
\end{prop}
\begin{proof}
By Proposition~\ref{dividing}, it suffices to show that $\ind[M]$ satisfies extension. Suppose we have $A\ind[M]_C B$, and let $B'$ be another set. We may assume that $C\subseteq A$, $C\subseteq B \subseteq B'$, and $A$, $C$, $B$ and $B'$ are algebraically closed. We would like to show that there exists $A'\equiv_B A$ such that $A'\ind[M]_C B'$.

Let $D = \langle AB\rangle \otimes_B B'$, the fibered coproduct of $\langle AB\rangle$ and $B'$ over $B$ in the category of $L$-structures. We can give an explicit description of $D$: Let $D_0$ be the disjoint union of $\langle AB\rangle$ and $B'$ over $B$, i.e.\ with the elements of $B$ in $\langle AB\rangle$ and in $B'$ identified. Any term with parameters from $D_0$ can be uniquely simplified with respect to $\langle AB\rangle$ and $B'$, by iteratively replacing any function symbol whose arguments are all elements of $\langle AB\rangle$, or whose  arguments are all elements of $B'$, by its value (as usual, we view constant symbols as $0$-ary functions). Then the underlying set of $D$ is given by the simplified terms with parameters from $D_0$, i.e.\ those terms with the property that no function symbol appearing in the term has all its argument in $\langle AB\rangle$ or all its arguments in $B'$. The interpretation of function symbols in $D$ is the obvious one (compose the function with the given simplified terms, then simplify if necessary), and the only instances of relations which hold in $D$ are those which hold in $\langle AB\rangle$ or in $B'$. 

Note that $\langle AB\rangle$ and $B'$ embed in $D$, by sending an element $a$ to the term $a\in D_0$, and $\langle AB\rangle \cap B' = B$ in $D$. Identifying these structures with their isomorphic copies in $D$, there is an embedding $i\colon D\to \mathbb{M}$ which is the identity on $B'$. Let $A' = i(A)$. Then $\langle A'B\rangle$ is isomorphic to $\langle AB\rangle$, so $A'\equiv_B A$. In particular, $A'\ind[M]_C B$. Of course, $\langle A'B'\rangle$ is isomorphic to $\langle A'B\rangle \otimes_B B'$.

Towards showing that $A'\ind[M]_C B'$, pick $C'$ with $C\subseteq C'\subseteq B'$. We may assume that $C'$ is algebraically closed. Let $\widetilde{C} = C'\cap B$, which is also algebraically closed. We will prove by induction on terms with parameters from $A'C'$ that:
\begin{enumerate}
\item If such a term evaluates to an element of $B'$, then that element is in $C'$.
\item If such a term evaluates to an element of $\langle A'B\rangle$, then that element is in $\langle A'\widetilde{C}\rangle$.
\end{enumerate}

First, we handle the base cases. The constant symbols are automatically in $C'$ and in $\langle A'\widetilde{C}\rangle$. The parameters from $C'$ are already in $C'$, and if they are also in $\langle A'B\rangle$, then since $C'\subseteq B'$, they are in $\langle A'B\rangle \cap B' = B$, and hence in $\widetilde{C}\subseteq \langle A'\widetilde{C}\rangle$. On the other hand, the parameters from $A'$ are already in $\langle A'\widetilde{C}\rangle$, and if they are also in $B'$, then they are in $\langle A'B\rangle \cap B' = B$, hence in $A'\cap B = C$ (as $A'\ind[a]_C B$), and $C\subseteq C'$. 

So suppose our term is $f(t_1,\dots,t_n)$, where $t_1,\dots,t_n$ are terms with parameters from $A'C'$ satisfying (1) and (2). Suppose $t_i$ evaluates to $c_i$ for all $i$, and let $b = f(c_1,\dots,c_n)$.

Case 1: $c_i\in B'$ for all $i$. Then by induction, $c_i\in C'$ for all $i$, and hence $b\in C'$. And if $b$ is also in $\langle A'B\rangle$, then by the argument for parameters in $C'$, it is in $\langle A'\widetilde{C}\rangle$.

Case 2: $c_i\in \langle A'B\rangle$ for all $i$. Then by induction, $c_i\in \langle A'\widetilde{C}\rangle$ for all $i$. (2) is immediate, since also $b\in \langle A'\widetilde{C}\rangle$. For (1), suppose $b\in B'$. Then $b\in \langle A'B\rangle\cap B' = B$. But since $A'\ind[M]_C B$, we have $A'\ind[a]_{\widetilde{C}} B$, so $b\in \widetilde{C}\subseteq C'$.

Case 3: Neither of the above. Then writing each $c_i$ in its normal form as a simplified term in $\langle A'B\rangle \otimes B'$, the element $b$ does not simplify down to a single parameter from $\langle A'B\rangle$ or from $B'$, since it is not the case that all of the arguments of $f$ come from $\langle A'B\rangle$ or from $B'$. So (1) and (2) are trivially satisfied.

Of course, condition (1) establishes that $A'\ind[a]_{C'} B'$, as desired.
\end{proof}

\begin{rem}\label{rem:mixedtrans}
If we define a new relation $\ind[\otimes]$ on subsets of $\mathbb{M}$ by $A\ind[\otimes]_C B$ if and only if the natural map $\langle AC\rangle \otimes_{\langle C\rangle} \langle  BC\rangle\to \langle ABC\rangle$ is an isomorphism, then we can interpret the proof of Proposition~\ref{extension} as a ``mixed transitivity'' statement. If $C\subseteq B \subseteq B'$, then: $$A\ind[d]_C B \text{ and }A\ind[\otimes]_B B' \implies A\ind[d]_C B'.$$
Thanks to this form of transitivity, $\ind[d]$ inherits extension from $\ind[\otimes]$.

It may be worth noting that a similar pattern occurs in Conant's analysis of forking and dividing in the theory $T_n$ of the generic $K_n$-free graph~\cite{Conant17}, which has $\SOP{3}$ (and hence also has $\SOP{1}$) when $n\geq 3$. Conant defines a relation $\ind[R]$ which satisfies extension (\cite[Lemma 5.2]{Conant17}), and the proof of \cite[Theorem 5.3]{Conant17} shows that $\ind[d]$ and $\ind[R]$ enjoy the same mixed transitivity property: $$A\ind[d]_C B \text{ and } A\ind[R]_B B' \implies A\ind[d]_C B'.$$
As a consequence, $\ind[d]$ inherits extension from $\ind[R]$, and hence $\ind[d] = \ind[f]$ in $T_n$.
\end{rem}

Proposition~\ref{extension} tells us that forking equals dividing for complete types. On the other hand, forking does not equal dividing for formulas, even over models.

\begin{prop}\label{prop:forkdiv}
Suppose $L$ contains an $n$-ary function symbol $f$ with $n\geq 2$. For any set $A$, there is a formula which forks over $A$ but does not divide over $A$.
\end{prop}
\begin{proof}
Let $b = (b_{0},\ldots, b_{n-1})$ be an $(n-1)$-tuple such that $b_{0}\notin \langle A\rangle$, and let $c$ be an element such that $c\notin \langle Ab\rangle$. Then the formula $\varphi(x;b,c)$ given by $(f(x,b) = c)\lor (x = b_{0})$ forks over $A$ but does not divide over $A$.

First, we show that the subformulas $f(x,b) = c$ and $x = b_{0}$ divide over $A$. For the first, let $(d_ic_i)_{i<\omega}$ be any sequence of realizations of $\tp(bc/A)$ such that $d_i = b$ for all $i<\omega$, but $c_i\neq c_j$ for all $i\neq j$ (this is possible, since $c\notin \acl(Ab)$). Then $\{f(x,d_i) = c_i\mid i<\omega\}$ is $2$-inconsistent. For the second, let $(e_i)_{i<\omega}$ be a sequence of realizations of $\tp(b_0/A)$ with $e_i\neq e_j$ for all $i\neq j$ (this is possible, since $b_0\notin \acl(A)$). Again, $\{x = e_i\mid i<\omega\}$ is $2$-inconsistent. Since $\varphi(x;b,c)$ is a disjunction of two formulas which divide over $A$, it forks over $A$.

To show that this formula does not divide, let $(d_ic_i)_{i<\omega}$ be any $A$-indiscernible sequence with $d_0c_0 = bc$. Write $d_i = (d_{i,0},\dots,d_{i,n-1})$.  If $d_{i,0} = b_0$ for all $i<\omega$, $\{\varphi(x;d_i,c_i)\mid i<\omega\}$ is consistent, since it is satisfied by $b_0$ itself. If not, then $d_{i,0}\neq d_{j,0}$ for all $i\neq j$, and we can find some $a$ such that $f(a,d_i) = c_i$ for all $i<\omega$, so $\{\varphi(x;d_i,c_i)\mid i<\omega\}$ is consistent in this case too.
\end{proof}

\subsection{Elimination of imaginaries}

\begin{defn}
The theory $T$ has \emph{weak elimination of imaginaries} if for all imaginary elements $e$, there is a real element $a \in \text{acl}^{\text{eq}}(e)$ with $e \in \text{dcl}^{\text{eq}}(a)$.  
\end{defn}

We prove weak elimination of imaginaries for $T^\emptyset_L$. The argument follows the standard route to elimination of imaginaries via an independence theorem as in \cite[Proposition 3.1]{hrushovski1991pseudo} and \cite[Subsection 2.9]{chatzidakis1998generic}.  

\begin{prop}\label{wei}
$T^\emptyset_L$ has weak elimination of imaginaries.
\end{prop}

\begin{proof}
Suppose we are given an imaginary element $e$, and suppose $a$ is a tuple from $\mathbb{M}$ and $f$ is a $0$-definable function (in $\mathbb{M}^{eq}$) with $f(a) = e$.  Put $C = \text{acl}^{\text{eq}}(e) \cap \mathbb{M}$ and $q = \text{tp}(a/C)$.  We may assume $\text{tp}(a/\text{acl}^{eq}(e))$ is not algebraic, because, if it is, we're done.  

\textbf{Claim}:  There are $a,b \models q$ with $a \ind[a]_{C} b$ with $f(a) = f(b) = e$.  

\emph{Proof of claim}:  $\text{tp}(a/\text{acl}^{\text{eq}}(e))$ is non-algebraic so, by extension, we can find $b \models \text{tp}(a/\text{acl}^{\text{eq}}(e))$ with $b \ind[a]_{\text{acl}^{\text{eq}}(e)} a$ in $\mathbb{M}^{\text{eq}}$. Note that also $f(b) = e$. Since $\text{acl}^{\text{eq}}(a) \cap \text{acl}^{\text{eq}}(b) = \text{acl}^{\text{eq}}(e)$, by intersecting with $\mathbb{M}$, we obtain $\text{acl}(a) \cap \text{acl}(b) = C$, that is, $a \ind[a]_{C} b$.  \qed 

Let $a,b$ be given as in the claim.  If $e$ is definable over $C$, we are done.  If $e$ is not definable over $C$, there is $e' \models \text{tp}(e/C)$ with $e' \neq e$ and we can find $c',d' \models q$ with $c' \ind[a]_{C} d'$ and $f(c') = f(d') = e'$, again by the claim.  As $c' \ind[a]_{C} d'$ we may, by extension, choose $c \equiv_{Cd'} c'$ with $c \ind[a]_{C} ad'$.  In particular, this gives $c \ind[a]_{C} a$ and $f(c) \neq f(a)$.  

As $a,c \models q$ we have $a \equiv_{C} c$.  Moreover, we have $a \ind[a]_{C} b$ and $c \ind[a]_{C} a$ so there is $a_{*} \models \text{tp}(a/ Cb) \cup \text{tp}(c/Ca)$ by the independence theorem (Theorem~\ref{thm:algind}).  Then we have $f(a_{*}) = f(b) = f(a) \neq f(a_{*})$, a contradiction.  
\end{proof}

\begin{rem}
$T^\emptyset_L$ does not eliminate imaginaries, since it does not even code unordered pairs. That is, there is no definable binary function $f(x,y)$ such that $f(a,b) = f(c,d)$ if and only if $\{a,b\} = \{c,d\}$. To see this, note that, by quantifier elimination, every definable function is defined piecewise by terms. Let $F_2$ be the $L$-structure freely generated over $\langle \emptyset\rangle$ (the substructure generated by the constants) by two elements, $a$ and $b$. Then if $t$ is a term, considered in the variable context $\{x,y\}$, such that $t(a,b) = t(b,a)$, then $t$ does not mention the variables, i.e.\ $t$ evaluates to an element of $\langle \emptyset\rangle$. For any copy of $F_2$ embedded in $\mathbb{M}$, $\tp(a,b) = \tp(b,a)$, since the automorphism of $F_2$ swapping $a$ and $b$ extends to an automorphism of $\mathbb{M}$. So any function $f$ coding unordered pairs must be defined by the same term $t$ on $(a,b)$ and on $(b,a)$. Then $t(a,b) = t(b,a)\in \langle \emptyset\rangle$. But this is a contradiction, since there are automorphisms of $\mathbb{M}$ which do not fix $\{a,b\}$ setwise.
\end{rem}

\section{Generic expansion and Skolemization}\label{sec:expsko}

\subsection{The theories $T_{\Sk}$ and $T_{L'}$}

\begin{defn}
Given a language $L$, define the language $L_{\Sk}$ by adding to $L$, for each formula $\varphi(x;y)$ with $l(x) = 1$, an $l(y)$-ary function symbol $f_{\varphi}$.  The \emph{Skolem expansion of }$T$ is the $L_{\Sk}$-theory $T_{+}$ defined by 
$$
T_{+} = T \cup \{\forall y\, (\exists x\, \varphi(x;y) \to \varphi(f_{\varphi}(y);y)) \mid \varphi(x;y) \in L, l(x) = 1\}.
$$
\end{defn}

Note that the Skolem expansion of $T$ contains Skolem functions for every formula of $L$, but does \emph{not} contain Skolem functions for every formula of $L_{\Sk}$.  

\begin{fact}\label{winkler}
Suppose $T$ is a model complete theory in the language $L$ that eliminates the quantifier $\exists^{\infty}$.  
\begin{enumerate}
\item The Skolem expansion $T_{+}$ of $T$ has a model completion $T_{\Sk}$ \cite[Theorem 2]{winkler1975model}.  We will refer to the theory $T_{\Sk}$ as the \emph{generic Skolemization of }$T$.
\item For any language $L'$ containing $L$, $T$, considered as an $L'$-theory, has a model companion $T_{L'}$ \cite[Theorem 5]{winkler1975model}.  We will refer to the theory $T_{L'}$ as the \emph{generic} $L'$\emph{-expansion of }$T$.
\end{enumerate}
\end{fact}

\begin{fact} \label{stillreasonable} \cite[Corollary 3 to Theorem 4]{winkler1975model}
Under the hypotheses of Fact~\ref{winkler}, $T_{\Sk}$ also eliminates the quantifier $\exists^{\infty}$.
\end{fact}

\subsection{Preservation of $\NSOP{1}$}

In this subsection, suppose $T$ is a fixed model complete theory in the language $L$ that eliminates the quantifier $\exists^{\infty}$, and let $L'$ be an arbitrary language containing $L$.  We may choose monster models $\mathbb{M}_{\text{Sk}} \models T_{Sk}$ and $\mathbb{M}_{L'} \models T_{L'}$. We may assume that both monster models have a common reduct to a monster model $\mathbb{M} \models T$. Note that we do not assume $T$ is complete, so reasoning in $\mathbb{M}$ amounts to working in an arbitrary completion of $T$.    

\begin{defn}
Let $\mathbb{M}^{*}$ denote either $\mathbb{M}_{\text{Sk}}$ or $\mathbb{M}_{L'}$ and let $L^{*}$ denote the corresponding language, either $L_{\text{Sk}}$ or $L'$.  For $a,b \in \mathbb{M}^{*}$ and $M \prec_{L^{*}} \mathbb{M}^{*}$, define
$$
a \ind[*]_{M} b \iff \text{acl}_{L^{*}}(Ma) \ind[K]_{M} \text{acl}_{L^{*}}(Mb) \text{ in }\mathbb{M}.
$$
\end{defn}

\begin{thm}\label{thm:algpres}
Let $\mathbb{M}^{*}$ denote either $\mathbb{M}_{\text{Sk}}$ or $\mathbb{M}_{L'}$ and let $L^{*}$ denote the corresponding language, either $L_{\text{Sk}}$ or $L'$.  If $T$ is $\NSOP{1}$, 
then $\ind[*]$ satisfies the independence theorem.
\end{thm}
\begin{proof}
We're given $M\prec\mathbb{M}^{*}$ and tuples $a,a',b,c$, with $a\ind[*]_M b$, $a'\ind[*]_M c$, $b\ind[*]_M c$, and $a\equiv_M a'$. Let $x_M$ be a tuple enumerating $M$, let $x_a$, $x_{a'}$, $x_b$ and $x_c$ be tuples enumerating $\acl_{L^{*}}(Ma)\setminus M$, $\acl_{L^{*}}(Ma')\setminus M$, $\acl_{L^{*}}(Mb)\setminus M$, and $\acl_{L^{*}}(Mc)\setminus M$, respectively, and let $x_{ab}$, $x_{a'c}$, and $x_{bc}$ be tuples enumerating $\acl_{L^{*}}(Mab)\setminus Mx_ax_b$, $\acl_{L^{*}}(Ma'c)\setminus Mx_{a'}x_c$, and $\acl_{L^{*}}(Mbc)\setminus Mx_bx_c$, respectively. 

We have $x_a\ind[K]_M x_b$, $x_{a'}\ind[K]_M x_c$, and $x_b\ind[K]_M x_c$ in $\mathbb{M}$. By the algebraically reasonable independence theorem in $T$, we can find $x_{a''}$ in $\mathbb{M}$ such that $x_{a''}\equiv^L_{Mx_b} x_a$, $x_{a''}\equiv^L_{Mx_c} x_{a'}$, $x_{a''}\ind[K]_M x_bx_c$, and further $x_{a''}\ind[a]_{Mx_b} x_c$, $x_{a''}\ind[a]_{Mx_c} x_b$, and $x_b\ind[a]_{Mx_{a''}} x_c$. By algebraically reasonable extension, at the expense of moving $x_{a''}$ by an automorphism in $\Aut(\mathbb{M}/Mx_bx_c)$, we may assume that $x_{a''}\ind[K]_M x_bx_cx_{bc}$ and $x_{a''}\ind[a]_{Mx_bx_c}x_{bc}$.

Pick an automorphism $\sigma\in \Aut(\mathbb{M}/Mx_b)$ moving $x_a$ to $x_{a''}$, and set $x_{a''b} = \sigma(x_{ab})$, so $x_{a''}x_{a''b}\equiv^L_{Mx_b} x_ax_{ab}$. Similarly, pick an automorphism $\sigma'\in \Aut(\mathbb{M}/Mx_c)$ moving $x_{a'}$ to $x_{a''}$, and set $x_{a''c} = \sigma'(x_{a'c})$, so $x_{a''}x_{a''c}\equiv^L_{Mx_c} x_{a'}x_{a'c}$. Now there are subtuples $y_{a''b}\subseteq x_{a''b}$ and $y_{a''c}\subseteq x_{a''c}$ enumerating $\acl_L(Mx_{a''}x_b)\setminus Mx_{a''}x_b$ and $\acl_L(Mx_{a''}x_c)\setminus Mx_{a''}x_c$, respectively. The algebraic independencies obtained so far imply that the tuples $x_M$, $x_{a''}$, $x_b$, $x_c$, $y_{a''b}$, $y_{a''c}$, and $x_{bc}$ are pairwise disjoint. 

By two applications of extension for algebraic independence over algebraically closed bases, we can find tuples $z_{a''b}$ and $z_{a''c}$ such that $z_{a''b}\equiv^L_{Mx_{a''}x_by_{a''b}} x_{a''b}$ and $z_{a''c}\equiv^L_{Mx_{a''}x_cy_{a''c}} x_{a''c}$, and so that the tuples $x_M$, $x_{a''}$, $x_b$, $x_c$, $z_{a''b}$, $z_{a''c}$, and $x_{bc}$ are pairwise disjoint. 

Let $\widehat{M}\prec \mathbb{M}$ be a small model of $T$ containing all these tuples.  We will expand $\widehat{M}$ to an $L^*$-structure, in order to embed it in $\mathbb{M}^{*}$. Let $p_{a''b} = \diagf(\acl_{L^{*}}(Mab))$, $p_{a''c} = \diagf(\acl_{L^{*}}(Ma'c))$, and $p_{bc} = \diagf(\acl_{L^{*}}(Mbc))$. We define interpretations of the relations, functions, and constants of $L^{*}$ according to $$p_{a''b}(x_M,x_{a''},x_b,z_{a''b})\cup p_{a''c}(x_M,x_{a''},x_c,z_{a''c})\cup p_{bc}(x_M,x_b,x_c,x_{bc}),$$ for the tuples that these diagrams refer to. This is consistent, since $p_{a''b}$, $p_{a''c}$, and $p_{bc}$ agree on $\diagf(\acl_{L^{*}}(Ma)) = \diagf(\acl_{L^{*}}(Ma'))$ (allowing $x_{a''}$ to enumerate both $\acl_{L^{*}}(Ma)\setminus M$ and $\acl_{L^{*}}(Ma')\setminus M$), $\diagf(\acl_{L^{*}}(Mb))$, and $\diagf(\acl_{L^{*}}(Mc))$. In the case that $L^{*} = L_{\text{Sk}}$, we observe that the values of the functions specified by these diagrams really give Skolem functions, since we have preserved the underlying $L$-types of all the tuples. For tuples not referred to by these diagrams, we define the interpretations of the relations and functions arbitrarily, taking care in the case that $L^* = L_{Sk}$ to satisfy the Skolem axioms (this is always possible, since $\widehat{M}$ is a model). 

Having expanded $\widehat{M}$ to an $L^{*}$-structure, we can embed it in $\mathbb{M}^{*}$ by $i\colon \widehat{M}\to \mathbb{M}^{*}$. Further, we may assume that $i$ is the identity on $(x_M,x_b,x_c,x_{bc})$, since $$\qftp_{\widehat{M}}^{L^{*}}(x_M,x_b,x_c,x_{bc}) = \qftp_{\mathbb{M}^{*}}^{L^{*}}(x_M,x_b,x_c,x_{bc}).$$ Let $a''$ be the subtuple of $i(x_M,x_{a''})$ corresponding to the subtuple of $(x_M,x_a)$ enumerating $a$.

Now $\qftp(x_M,i(x_{a''}),x_b,i(z_{a''b})) = \qftp(x_M,x_a,x_b,x_{ab})$, so $a''\equiv_{Mb} a$, and similarly $\qftp(x_M,i(x_{a''}),x_c,i(z_{a''c})) = \qftp(x_M,x_{a'},x_c,x_{a'c})$, so $a''\equiv_{Mc} a'$. And $a''\ind[*]_M bc$, since $i(x_{a''})$ (which enumerates $\acl_{L^{*}}(Ma'')\setminus M$) is Kim-independent over $M$ from $x_bx_cx_{bc}$ (which enumerates $\acl_{L^{*}}(Mbc)\setminus M$).
\end{proof}

\begin{cor}
Suppose $T$ is $\NSOP{1}$ 
Then:
\begin{enumerate}
\item  $T_\Sk$ is $\NSOP{1}$,
and $\ind[K] = \ind[*]$.
\item $T_{L'}$ is $\NSOP{1}$,
and $\ind[K] = \ind[*]$.  
\end{enumerate}
\end{cor}
\begin{proof}
We use Theorem~\ref{criterion}.
By Theorem~\ref{thm:algpres}, $\ind[*]$ satisfies the independence theorem. Existence over models, monotonicity, and symmetry follow immediately from the definition and the corresponding properties of $\ind[K]$ in the reduct.

Now suppose $a\nind[*]_M b$, witnessed by $a'\in \acl_{L^*}(Ma)$ and $b'\in \acl_{L^*}(Mb)$ such that $a'\nind[K]_M b'$ in the reduct. Let $\varphi(x';b',m)$ be the $L$-formula given strong finite character for $\ind[K]$ in the reduct. Let $\chi(x',a,m_1)$ isolate $\tp_{L^*}(a'/Ma)$, and let $\psi(y',b,m_2)$ isolate $\tp_{L^*}(b'/Mb)$. Then the formula $\varphi'(x;b,m,m_1,m_2)$ given by $\exists x'\,\exists y'\,(\chi(x',x,m_1)\land \psi(y',b,m_2)\land \varphi(x',y',m))$ gives strong finite character for $\ind[*]$. Similarly, if $\varphi$ gives witnessing for $\ind[K]$ in the reduct, the same formula $\varphi'$ gives witnessing for $\ind[*]$.
\end{proof}

\begin{rem}
By \cite[Lemma 3.1]{nubling2004adding}, $T_{\Sk}$ will never be a simple theory.  
\end{rem}

\subsection{Iterating to get built-in Skolem functions}

\begin{defn}
An $L$-theory $T$ has built-in Skolem functions if, for every formula $\varphi(x;y) \in L$ with $l(x)=1$, there is a definable $l(y)$-ary function $f_{\varphi}$ so that 
$$
T \models \forall y\, ( \exists x\, \varphi(x;y) \to \varphi(f_{\varphi}(y);y) ).
$$
\end{defn}

\begin{cor} \label{builtin}
Any $\NSOP{1}$ theory $T$ which eliminates $\exists^\infty$ has an expansion to an $\NSOP{1}$ theory $T_{\Sk}^\infty$ in a language $L_{\Sk}^\infty$ with built-in Skolem functions. Moreover, if $\mathbb{M}_{\Sk}^\infty$ is a monster model for $T_{\Sk}^\infty$, and $\mathbb{M}$ is its reduct to $L$, then for every $M\prec \mathbb{M}_{\Sk}^\infty$ and tuples $a$ and $b$, $$a\ind[K]_M b \text{ in $\mathbb{M}_{\Sk}^\infty$} \iff \acl_{L_{\Sk}^\infty}(Ma) \ind[K]_M \acl_{L_{\Sk}^\infty}(Mb)\text{ in }\mathbb{M}.$$
\end{cor}
\begin{proof}
We define $T_{\Sk}^\infty$ by induction. Let $T_0$ be the Morleyization of $T$ in the expanded language $L_0$, so $T_0$ is model complete. Given $T_n$, which we may assume by induction to be a model complete $\NSOP{1}$ theory which eliminates $\exists^\infty$, let $L_{n+1} = (L_n)_{\Sk}$, and let $T_{n+1} = (T_n)_{\Sk}$. Then, by Fact~\ref{stillreasonable} and Theorem~\ref{thm:algpres}, $T_{n+1}$ is again a model complete $\NSOP{1}$ theory which eliminates $\exists^\infty$ and has Skolem function for formulas in the language of $T_n$. And by Theorem~\ref{thm:algpres} and induction, if $\mathbb{M}_{n+1}$ is a monster model for $T_{n+1}$, and $\mathbb{M}$ is its reduct to $L$,
\begin{align*}
a\ind[K]_M b \text{ in $\mathbb{M}_{n+1}$} &\iff \acl_{L_{n+1}}(Ma) \ind[K]_M \acl_{L_{n+1}}(Mb)\text{ in }\mathbb{M}_n\\
&\iff \acl_{L_{n+1}}(Ma) \ind[K]_M \acl_{L_{n+1}}(Mb)\text{ in }\mathbb{M},
\end{align*}
since a set which is $\acl_{L_{n+1}}$-closed is also $\acl_{L_n}$-closed.

Then the theory $T^\infty_{\Sk} = \bigcup_{n<\omega} T_n$ is $\NSOP{1}$ and has built-in Skolem functions. It remains to lift the characterization of Kim-independence to this theory. 

If $\varphi(x;b)$ is a formula in $L^{\infty}_{\Sk}$, possibly with parameters in $M$, there is some $n$ so that $\varphi(x;y) \in L_{n}$.  Let $I = (b_{i})_{ < \omega}$ be an $M$-finitely satisfiable sequence in $\text{tp}_{L^{\infty}_{\Sk}}(b/M)$.  Then $I$ is also an $M$-finitely satisfiable sequence in $\text{tp}_{L_{n}}(b/M)$.  By Kim's lemma for Kim-dividing \cite[Theorem 3.15]{KRKim}, $\varphi(x;b)$ Kim-divides over $M$ if and only if $\{\varphi(x;b_{i}) \mid i < \omega\}$ is consistent, hence $\varphi(x;b)$ Kim-divides over $M$ in $\mathbb{M}_{n}$ if and only if $\varphi(x;b)$ Kim-divides over $M$ in $\mathbb{M}^{\infty}_{\Sk}$.  It follows that $a \ind[K]_{M} b$ in $\mathbb{M}^{\infty}_{\Sk}$ if and only if $a \ind[K]_{M} b$ in $\mathbb{M}_{n}$ for all $n$.

Now we show $a \ind[K]_{M} b$ in $\mathbb{M}^{\infty}_{\Sk}$ if and only if $\text{acl}_{L^{\infty}_{\Sk}}(Ma) \ind[K]_{M} \text{acl}_{L^{\infty}_{\Sk}}(Mb)$ in $\mathbb{M}$.  If $\text{acl}_{L^{\infty}_{\Sk}}(Ma) \nind^{K}_{M} \text{acl}_{L^{\infty}_{\Sk}}(Mb)$ in $\mathbb{M}$, then there is $n$ so that $\text{acl}_{L_{n}}(Ma) \nind^{K}_{M} \text{acl}_{L_{n}}(Mb)$ hence there is some formula $\varphi(x;b) \in \text{tp}_{L_{n}}(a/Mb)$ that Kim-divides over $M$ in $\mathbb{M}_{n}$.  This formula witnesses $a \nind^{K}_{M} b$ in $\mathbb{M}^{\infty}_{\Sk}$.  Conversely, if $a \nind^{K}_{M} b$ in $\mathbb{M}^{\infty}_{\Sk}$, then there is some formula $\varphi(x;b) \in \text{tp}_{L^{\infty}_{\Sk}}(a/Mb)$ witnessing this.  Then for some $n$, $\varphi(x;b) \in \text{tp}_{L_{n}}(a/Mb)$ and this formula witnesses $a \nind^{K}_{M} b$ in $\mathbb{M}_{n}$.  It follows that $\text{acl}_{L_{n}}(Ma) \nind^{K}_{M} \text{acl}_{L_{n}}(Mb)$ in $\mathbb{M}$, and therefore $\text{acl}_{L^{\infty}_{\Sk}}(Ma) \nind^{K}_{M} \text{acl}_{L^{\infty}_{\Sk}}(Mb)$ by monotonicity.  
\end{proof}

\subsection*{Acknowledgements}
This work constitutes part of the dissertation of the second-named author, advised by Thomas Scanlon.  We wish to thank Artem Chernikov for pointing out the results of Herwig N\"ubling mentioned above and Emil Je\v{r}\`{a}bek for sharing his preprint with us and allowing us to use it here.

\bibliographystyle{alpha}
\bibliography{ms.bib}{}

\end{document}